\documentclass[12pt,a4paper]{amsart}
\usepackage{a4wide}
\usepackage{amsmath, amssymb, amsfonts,enumerate}
\usepackage[all]{xy}
\usepackage{amscd}
\usepackage{comment}
\usepackage{mathtools}
\usepackage{tikz-cd} 
\usepackage{url}

\newtheorem{theorem}{Theorem}[section]
\newtheorem{lemma}[theorem]{Lemma}
\newtheorem{corollary}[theorem]{Corollary}
\newtheorem{proposition}[theorem]{Proposition}

 \theoremstyle{definition}
 \newtheorem{definition}[theorem]{Definition}
 \newtheorem{remark}[theorem]{Remark}

 \newtheorem{example}[theorem]{Example}

\newtheorem{question}{Question}

\numberwithin{equation}{section}
\newcommand {\K}{\mathbb{K}} 
\newcommand {\N}{\mathbb{N}} 
\newcommand {\Z}{\mathbb{Z}} 
\newcommand {\R}{\mathbb{R}} 
\newcommand {\C}{\mathbb{C}} 

\newcommand{\CC}{\mathcal{C}}

\newcommand{\FF}{\mathcal{F}}

\newcommand{\OO}{\mathcal{O}}

\newcommand{\A}{\mathbb{A}} 
\newcommand{\Proj}{\mathbb{P}}

\DeclareMathOperator{\mdim}{mdim}

\DeclareMathOperator{\im}{Im}

\DeclareMathOperator{\Id}{Id}

\DeclareMathOperator{\Spec}{Spec}

\begin{document}
\title[On the Garden of Eden theorem]{On the Garden of Eden theorem for endomorphisms of symbolic algebraic varieties}
\author[T.Ceccherini-Silberstein]{Tullio Ceccherini-Silberstein}
\address{Universit\`a del Sannio, I-82100 Benevento, Italy}
\email{tullio.cs@sbai.uniroma1.it}
\author[M.Coornaert]{Michel Coornaert}
\address{Universit\'e de Strasbourg, CNRS, IRMA UMR 7501, F-67000 Strasbourg, France}
\email{michel.coornaert@math.unistra.fr}
\author[X.K.Phung]{Xuan Kien Phung}
\address{Universit\'e de Strasbourg, CNRS, IRMA UMR 7501, F-67000 Strasbourg, France}
\email{phung@math.unistra.fr}
\begin{abstract}
Let $G$ be an amenable group  and let $X$ be an irreducible complete algebraic variety over an algebraically closed field $K$. 
Let $A$ denote the set of $K$-points of $X$ and let
$\tau \colon A^G \to A^G$ be an algebraic cellular automaton over $(G,X,K)$, 
that is, a cellular automaton over the group $G$ and the alphabet $A$ whose local defining map is 
induced by a morphism of $K$-algebraic varieties. 
We introduce a weak notion of pre-injectivity for algebraic cellular automata, namely $(*)$-pre-injectivity, and prove that $\tau$  is surjective
if and only if it is $(*)$-pre-injective. 
In particular, $\tau$ has the Myhill property, i.e., 
is surjective whenever it is pre-injective.
Our result gives a positive answer to a question raised by Gromov in~\cite{gromov-esav} 
and yields an analogue of the classical Moore-Myhill Garden of Eden theorem. 

\end{abstract}
\date{\today}
\maketitle

\setcounter{tocdepth}{1}
\tableofcontents

\section{Introduction}

In~\cite{gromov-esav}, Gromov brought out fascinating  connections between algebraic geometry and symbolic dynamics.
At the end of~\cite{gromov-esav}, he asked the following:
\begin{quote}
8.J. Question. Does the Garden of Eden theorem generalize to the 
proalgebraic category?
First, one asks if pre-injective  $\implies$  surjective, while the reverse implication 
needs further modification of definitions.
\end{quote}
Our goal here is to present some positive answers to Gromov's question.
Before stating our main results,
we need to recall a few facts related to the classical Garden of Eden theorem, symbolic dynamics, and algebraic geometry
(see Section~\ref{sec:background} for more details and  references).    
\par
Fix a set $A$, called the \emph{alphabet},  and a group  $G$, called the \emph{universe}.
The set   $A^G \coloneqq  \{c \colon G \to A\}$,
consisting of all maps from $G$ to $A$, is called the set of \emph{configurations}. 
Equip $A^G$ with the $G$-\emph{shift}, i.e.,
the  action of  $G$  
defined by
the map $G \times A^G \to A^G$, $(g,c) \mapsto g c$, 
where $(gc)(h) \coloneqq  c(g^{-1}h)$ for all 
$g,h \in G$ and $c \in A^G$.
\par
Given a configuration $c \in A^G$ and a subset $\Omega \subset G$, we write $c\vert_\Omega$ for the restriction of $c$ to $\Omega$, i.e., the element $c\vert_\Omega \in A^\Omega$ defined by
$c\vert_\Omega(g) \coloneqq  c(g)$ for all $g \in \Omega$.
\par
A \emph{cellular automaton} over  the group $G$ and the alphabet $A$ is a map
$\tau \colon A^G \to A^G$ satisfying the following property:
there exist a finite subset $M \subset G$
and a map $\mu \colon A^M \to A$ such that 
\begin{equation} 
\label{e:local-property}
(\tau(c))(g) = \mu((g^{-1} c )\vert_M)  \quad  \text{for all } c \in A^G \text{ and } g \in G.
\end{equation}
Such a set $M$ is then called a \emph{memory set} of $\tau$ and $\mu$ is called the associated \emph{local defining map}
(see~\cite{ca-and-groups-springer}).
Note that it immediately follows from~\eqref{e:local-property} that every cellular automaton
$\tau \colon A^G \to A^G$ is $G$-\emph{equivariant}, i.e., satisfies
$\tau(g c) = g \tau(c)$ for all $c \in A^G$ and $g \in G$,
and continuous with respect to the prodiscrete topology, that is, the product topology on $A^G$ obtained by taking the discrete topology on each factor $A$ of $A^G$.   
 \par
Two configurations $c_1,c_2  \in A^G$ are said to be \emph{almost equal} if  the set  $\{g \in G : c_1(g) \not= c_2(g) \}$ is finite.
 A cellular automaton $\tau \colon A^G \to A^G$ is called \emph{pre-injective} if
 $\tau(c_1) = \tau(c_2)$ implies $c_1 = c_2$ whenever $c_1, c_2 \in A^G$ are almost equal.
 \par
 In 1963, Myhill~\cite{myhill} proved that if $A$ is a finite set and $G = \Z^d$ 
 ($d \in \N$), then every pre-injective cellular automaton
 $\tau \colon A^G \to A^G$ is surjective.
 Together with the converse implication, which had been established shortly before by Moore~\cite{moore}, this yields the celebrated Garden of Eden theorem of Moore and Myhill stating that a cellular automaton with finite alphabet over the group $\Z^d$ is pre-injective if and only if it is surjective.
 The Garden of Eden theorem was subsequently extended to cellular automata with finite alphabet over amenable groups in~\cite{ceccherini}.
 There is also a linear version of the Garden of Eden theorem.
 More precisely, it is shown 
  in~\cite{csc-garden-linear} (see also \cite[Theorem~8.9.6]{ca-and-groups-springer}) that if $A$ is a finite-dimensional vector space over a field $K$ and $G$ is an amenable group,
  then a $K$-linear cellular automaton $\tau \colon A^G \to A^G$ is pre-injective if and only if it is surjective.
 \par
Consider now an algebraic variety $X$ over a field $K$,
i.e., a scheme of finite type over $K$, 
and let $A \coloneqq X(K)$ denote the set of $K$-points of $X$, that is,
the set consisting of all $K$-scheme morphisms $\Spec(K) \to X$.
We say that a cellular automaton $\tau \colon A^G \to A^G$ is an \emph{algebraic cellular automaton}
over $(G,X,K)$
if $\tau$ admits a memory set $M$ 
with local defining map $\mu \colon A^M \to A$
such that $\mu$ is induced by some $K$-scheme morphism 
$f \colon X^M \to X$,
where $X^M$ denotes the $K$-fibered product of a family of copies of $X$ indexed by $M$
(cf.~Definition~1.1 in~\cite{cscp-alg-ca}).
\par
In the present paper, we shall first establish a version of the 
Myhill part of the Garden of Eden theorem for certain algebraic cellular automata.
This yields a positive answer to the first part of Gromov's question.
More specifically, we shall prove the following result (cf. Theorem~\ref{t:goe-alg} for a more general statement). 

\begin{theorem}
\label{t:main-theorem-1}
Let $G$ be an amenable group and let $X$ be an irreducible complete algebraic variety over an algebraically  closed field $K$. 
Let $A \coloneqq X(K)$ denote the set of $K$-points of $X$.
Then every pre-injective  algebraic cellular automaton $\tau \colon A^G \to A^G$
over $(G,X,K)$ is surjective.
\end{theorem}

As injectivity trivially implies pre-injectivity,  
an immediate consequence of Theorem~\ref{t:main-theorem-1} is the following.

\begin{corollary}
\label{c:main-theorem-inj}
Let $X$ be an irreducible  complete algebraic variety over an algebraically  closed field $K$ and let $G$ be an amenable group. 
Let $A \coloneqq X(K)$ denote the set of $K$-points of $X$.
Then every injective  algebraic cellular automaton $\tau \colon A^G \to A^G$
over $(G,X,K)$ is surjective.
\end{corollary}

It is shown in \cite[Theorem~1.2]{cscp-alg-ca}
that if $X$ is a complete (possibly not irreducible) algebraic variety over an algebraically closed field $K$ and $G$ is a locally residually finite group, then 
every injective algebraic cellular automaton  
over $(G,X,K)$ is surjective.
Therefore, Corollary~\ref{c:main-theorem-inj} remains true if the hypothesis that 
$G$ is \emph{amenable} is replaced by the hypothesis that $G$ is \emph{locally residually finite}.
We shall see in Example~\ref{ex:free-not-myhill} that if $G$ is a free group on two generators, then, given any algebraically closed field $K$, there exist  an irreducible  complete $K$-algebraic variety $X$ and 
an algebraic cellular automaton  over $(G,X,K)$ that is pre-injective but not surjective. As a free group on two generators is residually finite, we deduce  that 
Theorem~\ref{t:main-theorem-1} becomes false if \emph{amenable} is replaced by  
\emph{residually finite} in its hypotheses.
\par
Let us note that, as implicitly stated  in Gromov's question,  the converse implication, i.e., the analogue of the Moore implication, does not hold under the hypotheses of 
Theorem~\ref{t:main-theorem-1}.
For example, if $K$ is an  algebraically closed field  whose characteristic is not equal to $2$,
the projective line $\Proj_K^1$ is an irreducible complete $K$-algebraic variety
and the morphism $f \colon \Proj_K^1 \to \Proj_K^1$ given by $(x:y) \mapsto (x^2:y^2)$ is surjective but not injective. Taking $A \coloneqq \Proj_K^1(K) $, 
we deduce  that, for any group $G$,
the map $\tau \colon A^G \to A^G$ defined by 
$(\tau(c))(g) \coloneqq  f(c(g))$ for all $c \in A^G$ and $g \in G$,
is an algebraic cellular automaton over $(G,X,K)$ that is surjective but not 
pre-injective.     
\par
In order to formulate a version of the Garden of Eden theorem for algebraic cellular automata, we introduce a weak notion of pre-injectivity for them, namely $(*)$-\emph{pre-injectivity}
(see Definition~\ref{def:*-**} below).
We shall prove that Theorem~\ref{t:main-theorem-1}
 remains valid if we replace the hypothesis that $\tau$ is pre-injective by the weaker hypothesis that 
 $\tau$ is $(*)$-pre-injective. 
This weak form of pre-injectivity also allows us to establish a version of the Moore part of the Garden of Eden theorem for algebraic cellular automata. 

\begin{theorem}
\label{t:main-theorem-3}
Let $G$ be an amenable group and let $X$ be an irreducible algebraic variety over an algebraically closed field $K$. 
Let $A \coloneqq X(K)$ denote the set of $K$-points of $X$. 
Then every surjective algebraic cellular automaton $\tau \colon A^G \to A^G$
over $(G,X,K)$ is $(*)$-pre-injective.
\end{theorem}  

Note that $X$ is not assumed to be complete in Theorem~\ref{t:main-theorem-3}.
Combining these results, 
we obtain the following version of the Garden of Eden theorem 
(see Theorem~\ref{t:goe-alg}) for algebraic cellular automata.  

\begin{theorem}
\label{t:main-theorem-2}
Let $G$ be an amenable group and let $X$ be an irreducible complete algebraic variety over an algebraically  closed field $K$. 
Let $A \coloneqq X(K)$ denote the set of $K$-points of $X$
and let  $\tau \colon A^G \to A^G$ be an algebraic cellular automaton over $(G,X,K)$.
Then the following conditions are equivalent:
\begin{enumerate}[\rm(a)]
\item
$\tau$ is surjective;
\item
$\tau$ is $(*)$-pre-injective.
\end{enumerate} 
\end{theorem}

The paper is organized as follows. 
The next  section collects 
background material on algebraic varieties and amenable groups.
Section~\ref{sec:alg-ca} contains some preliminary results on
algebraic cellular automata. 
In Section~\ref{sec:mean-dim}, we introduce the \emph{algebraic mean dimension}
$\mdim_\FF(\Gamma)$ of a subset $\Gamma \subset A^G$, where $G$ is an amenable group equipped with a F\o lner net $\FF$ and
$A$ is the set of $K$-points of an algebraic variety $X$ over an algebraically closed field $K$.
The definition of algebraic mean dimension is anologous to that of topological entropy.
Here $\mdim_\FF(\Gamma)$ is obtained as a  limit of the average Krull dimension of the projection of $\Gamma$ along the F\o lner net. It follows in particular that
 $\mdim_\FF(\Gamma)$ is always bounded above by the dimension of the variety $X$
 and equality holds if $\Gamma = A^G$.
 In Section~\ref{sec:surjectivity}, we prove that if $X$ is irreducible and complete, then $\tau$ is surjective if and only if its image  has maximal algebraic mean dimension
 (Theorem~\ref{t:sur-dim}).
In Section~\ref{sec:weeak-pre-inj}, we introduce the notions of $(*)$-pre-injectivity and 
$(**)$-pre-injectivity, which are both implied by pre-injectivity.
In the trivial case when $A$ is finite, that is, $X$ is $0$-dimensional, 
every cellular automaton $\tau \colon A^G \to A^G$ is algebraic over $(G,X,K)$ and
$(*)$-pre-injectivity is equivalent to pre-injectivity (cf. Example~\ref{ex:reducible}).
We show that if $X$ is irreducible and complete,
then $\tau$ is $(*)$-pre-injective if and only if it is $(**)$-pre-injective
(see Assertion~(iii) in Proposition~\ref{p:pre-injectivity-*}).
We also establish relations between $(*)$-pre-injectivity, $(**)$-pre-injectivity, and the fact
that the image of $\tau$ has maximal algebraic mean dimension.
In Section~\ref{sec:main-results},
we combine the results of the two previous sections to obtain Theorem~\ref{t:main-theorem-3} and
Theorem~\ref{t:goe-alg},
which extends Theorem~\ref{t:main-theorem-1} as well as Theorem~\ref{t:main-theorem-2}.
Another result in this section says that, under suitable conditions,
the surjectivity  of  an algebraic cellular automaton, provided it is defined over an amenable group and an algebraically closed field, is a property that is invariant under  base change of the ground field 
(Theorem~\ref{t:base-change-surj}).
Several counterexamples are presented in Section~\ref{sec:counterexamples} showing that the hypotheses in our results are reasonably optimal. 
Some open questions are formulated    in the final section.

Let $G$ be an amenable group with a F\o lner net $\FF$. 
Let $X$ be an algebraic variety over an algebraically closed field $K$. 
Let $A\coloneqq X(K)$ and suppose that $\tau \colon A^G \to A^G$ is an algebraic cellular automaton over $(G,X,K)$. 
We can summarize our results in the following diagram relating the various properties of the algebraic cellular automaton $\tau$:
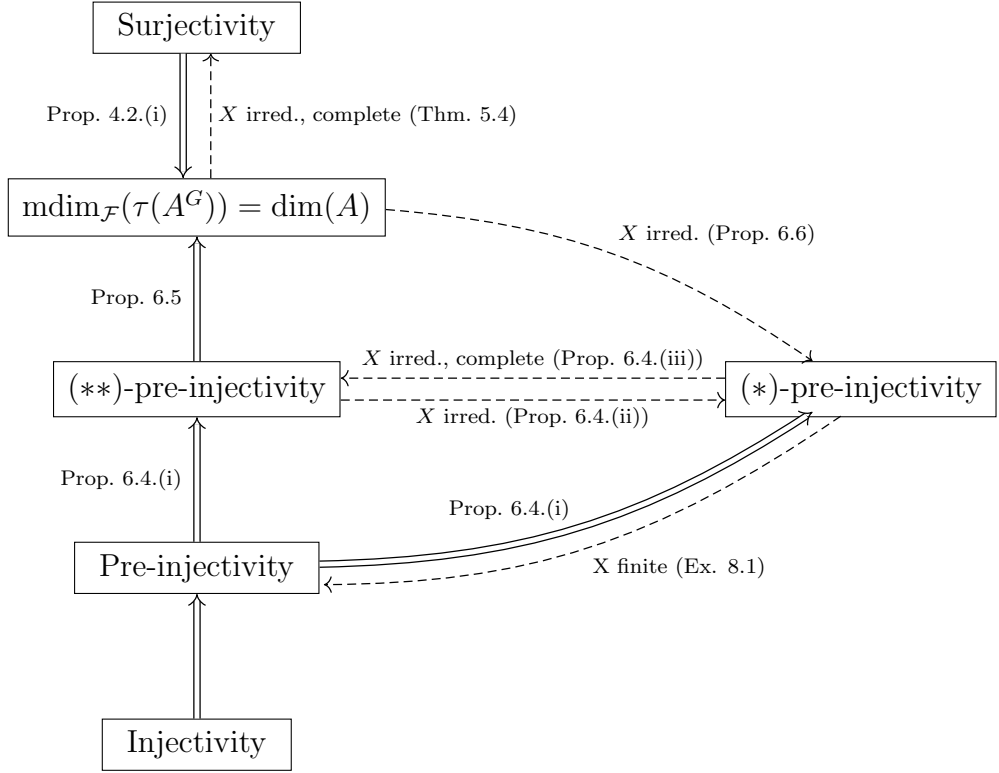
\[
\begin{tikzcd}
[cells={nodes={draw=black}}, column sep=2.7em, row sep=4em]
 & \text{ Surjectivity } 
 \arrow[d,"\text{Prop.~\ref{p:prop-faciles-mdim}.(i) }", Rightarrow, shift right=1ex, swap]  
 & &  & \\
 & \mdim_\FF(\tau(A^G))=\dim(A) 
 \arrow[u, "X \text{ irred., complete (Thm.~\ref{t:sur-dim})}", dashrightarrow, shift right=1ex, swap] 
 \arrow[d, "\text{Prop.~\ref{p:**-implies-mdim} }", Leftarrow, swap] 
 \arrow[drrrr, "X \text{ irred.} \text{ (Prop.~\ref{p:mdim-implies-*-pre})} ", dashrightarrow, bend left=15] 
&  & & &\\
& (**)\text{-pre-injectivity} 
\arrow[rrrr,"X\text{ irred.} \text{ (Prop.~\ref{p:pre-injectivity-*}.(ii)})", dashrightarrow, shift right=0.8ex, swap] 
& & & & (*)\text{-pre-injectivity}   
\arrow[llll, "X\text{ irred., complete} \text{ (Prop.~\ref{p:pre-injectivity-*}.(iii)})", dashrightarrow, shift right=0.8ex,swap]  
\arrow[dllll, "\text{X finite (Ex. \ref{ex:reducible})}", dashrightarrow,  shorten <= -0.8em , shift left=1.0ex, bend left=17] \\ 
& \text{ Pre-injectivity }   
\arrow[u, "\text{Prop.~\ref{p:pre-injectivity-*}.(i) }", Rightarrow] 
\arrow[urrrr, "\text{Prop.~\ref{p:pre-injectivity-*}.(i)}", Rightarrow, shorten <= 0.05em, shift left=0.4ex, bend right=16] 
& & & &\\
& \text{ Injectivity }  
\arrow[u, Rightarrow] & & & &
\end{tikzcd}
\]  

\section{Background material and preliminary results}
\label{sec:background}

\subsection{Krull dimension and Jacobson spaces}
Let $X$ be a topological space.
Given a subset $Y \subset X$, we denote by $\overline{Y}$ the closure of $Y$ in $X$.
\par
A point $x \in X$ is said to be  a \emph{closed} (respectively, \emph{generic}) point of $X$ 
if   $\overline{\{x\}} = \{x\}$  
(respectively, ~$\overline{\{x\}} = X$).  
\par
One says that $X$ is \emph{irreducible} if every non-empty open subset of $X$ is dense in $X$.
This amounts to saying that if $X = Y \cup Z$, where $Y$ and $Z$ are closed subsets of  $Y$, then 
$X = Y$ or $X = Z$. 
\par
A subset $Y \subset X$ is called an \emph{irreducible component} of $X$ if $Y$ is 
irreducible (for the induced topology) and maximal for inclusion among all irreducible subsets of $X$.
As the closure of an irreducible subset of $X$ is irreducible,
every irreducible component of $X$ is closed in $X$.
By Zorn's lemma, every irreducible subset of $X$ is contained in  some irreducible component of $X$.
Since every singleton of $X$ is irreducible, it follows that $X$ is the union of its irreducible components.
\par
The topological space $X$ is called \emph{Noetherian} if every descending chain of closed subsets of $X$ is stationary. 
Every subset of a Noetherian topological space is Noetherian for the induced topology.
If $X$ is Noetherian, then $X$ is quasi-compact and admits only finitely many irreducible components.
\par
The \emph{Krull dimension} of $X$, denoted by $\dim(X)$, is defined as being the supremum of the lengths of all the strictly ascending  chains of closed irreducible  subsets of $X$.

\begin{proposition}
\label{p:dim-subsets}
Let $X$ be a topological space.
Then the following hold:
\begin{enumerate}[\rm (i)]
\item
if $Y$ is a subset of $X$, then $\dim(Y) \leq \dim(X)$;
\item
if $X$ is irreducible with $\dim(X) < \infty$ and $Y$ is a closed subset of $X$ such that 
$\dim(Y) = \dim(X)$, 
then one has $Y = X$;
\item
if $(U_\lambda)_{\lambda \in \Lambda}$ is an open cover of $X$, then
one has $\dim(X) = \sup_{\lambda \in \Lambda} \dim(U_\lambda)$;
\item
one has $\dim(X) = \sup_{Y \in \CC(X)}  \dim(Y)$, where $\CC(X)$ denotes the set of all irreducible components of $X$;
\item
if $X$ is the union of a finite family $(Z_i)_{i \in I}$ of closed irreducible subsets of $X$, then every irreducible component of $X$ is equal to one of the $Z_i$ and one has
$\dim(X) = \max_{i \in I} \dim(Z_i)$.
 \end{enumerate} 
\end{proposition}

\begin{proof}
For (i), (ii), (iii), and (iv), see \cite[Lemma~5.7]{gorz}.
Assertion (v) follows from \cite[Proposition~2.4.5.(c)]{liu-alg-geom} and Assertion~(iii).
\end{proof}

A subset $Y$ of a topological space $X$ is said to be \emph{very dense} in $X$ if
$F \cap Y$ is dense in $F$  for every closed subset $F$ of $X$.

\begin{proposition}
\label{p:dim-very-dense}
Suppose that $Y$ is a very dense subset of a topological space $X$.
Then one has $\dim(X) = \dim(Y)$.
\end{proposition}

\begin{proof}
By Proposition~\ref{p:dim-subsets},
it suffices to prove that $\dim(X) \leq \dim(Y)$.
We first observe that if $F$ and $F'$ are closed subsets of $X$ such that $F \cap Y = F' \cap Y$, 
then $F = F'$ since $Y$ is very dense in $X$. 
Note  also that $F\cap Y$ is irreducible for every closed irreducible subset $F$ of $X$.  
Thus, if $F_0 \subset F_1 \subset \dots \subset F_n$ is a strictly ascending chain of closed irreducible subsets of $X$, then
$F_0 \cap Y \subset F_1 \cap Y \subset \dots \subset F_n \cap Y$ is a strictly ascending chain of closed irreducible subsets of $Y$.
It follows that $\dim(X) \leq \dim(Y)$.
\end{proof}

Let $X$ be a topological space.
We denote by $X_0$ the set of closed points of $X$.
One says that  the topological space  $X$ is  \emph{Jacobson} if  $X_0$ is very dense
in $X$.
From the result of Proposition~\ref{p:dim-very-dense}, we immediately deduce the following.

\begin{corollary}
\label{c:dim-closed-points-jacob}
Let $X$ be a Jacobson space. 
Then one has $\dim(X) = \dim(X_0)$.
\end{corollary}

 A subset of a topological space $X$  is said to be \emph{locally closed} if it is the intersection of an open subset and a closed subset of $X$.
 A subset  of  $X$ is called  \emph{constructible} if it is a finite union of locally closed subsets of $X$.
The set of constructible subsets of $X$ is closed 
under finite union, finite intersection, and set difference.
Every constructible subset  $C \subset X$ contains a dense open subset of $\overline{C}$
(see~\cite[Lemma~2.1]{an-rigid}).

\begin{proposition}
\label{p:construct-inJacob}
Let $X$ be a Jacobson topological space and let $C$ be a constructible subset of $X$.
Then the following hold:
\begin{enumerate}[\rm (i)]
\item
 $C$ is Jacobson;
 \item
 $C_0 = C \cap X_0$;
 \item
 $\dim(C) = \dim(C_0) = \dim(C_0 \cap X_0)$.
\end{enumerate}
\end{proposition}

\begin{proof}
See e.g.~\cite[Lemma~2.2]{cscp-alg-ca} for the proof of (i) and (ii).
Assertion~(iii) follows from (i),   (ii),  and Corollary~\ref{c:dim-closed-points-jacob}.
\end{proof}

As immediate consequences of the preceding proposition, we get the following.

\begin{corollary}
\label{c:constr-Jacob-bije}
Let $X$ be a Jacobson space.
Then the map $C \mapsto C \cap X_0$ yields a bijection from the set of constructible subsets of $X$ onto the set of constructible subsets of $X_0$.
Moreover, this map preserves Krull dimension.
\end{corollary}

\begin{corollary}
\label{c:jacob-irred}
Let $X$ be a Jacobson space. 
Then $X$ is irreducible if and only if $X_0$ is irreducible. 
\end{corollary}

\subsection{Schemes and algebraic varieties}
In this subsection, we collect all the material about schemes and algebraic  varieties
that we shall need in the present paper
(see  
\cite{gorz},
\cite{grothendieck-20-1964}, 
\cite{ega-4-3}, \cite{harris}, \cite{hartshorne}, 
\cite{liu-alg-geom},
\cite{vakil}
  for more details).
 All rings are commutative with $1$.
 We recall that a \emph{scheme} is a locally ringed space, that is, a topological space endowed with a sheaf of rings such that the stalk at each point is a local ring.
Following a common abuse, if there is no risk of confusion, we shall use the same symbol to denote a scheme and its underlying topological space. 
The topology on the underlying topological space of a scheme is called the \emph{Zariski topology}.    
\par
Every scheme $X$ is \emph{sobre}, i.e., the map $x \mapsto \overline{\{x\}}$ yields a bijection from $X$ onto the set of non-empty closed irreducible subsets of $X$ (see e.g.~Proposition~3.23 in~\cite{gorz}).
In particular, every non-empty closed irreducible subset of a scheme $X$ admits a unique generic point.  
A scheme is called \emph{irreducible} (respectively, \emph{Jacobson}) if its underlying topological space is irreducible (respectively, Jacobson).
The \emph{Krull dimension} $\dim(X)$ of a scheme $X$ is define as being the  dimension of its underlying topological space.
\par
The \emph{spectrum} of a ring $R$ is a scheme whose underlying set consists of all prime ideals of $R$. 
The spectrum of a ring $R$ is denoted by $\Spec(R)$ or simply $R$ when there is no risk of confusion.
The \emph{Krull dimension} $\dim(R)$ of a ring $R$ is the Krull dimension of its spectrum.
It is equal to the supremum of the lengths of all the strictly ascending chains of prime ideals of $R$. 
\par
A scheme $X$ is called \emph{Noetherian} if the space $X$ admits a finite affine open cover 
$(U_i)_{i \in I}$ such that, for each $i \in I$,
one has $U_i = \Spec(R_i)$,
where $R_i$ is a Noetherian ring.
The underlying topological space of every Noetherian scheme is  Noetherian.
However,   there are schemes that are not Noetherian although their underlying topological spaces are Noetherian.
\par
Let $K$ be a field. An \emph{algebraic variety} over $K$ (or $K$-\emph{algebraic variety}) is a scheme of finite type over $K$.
\par
Given an algebraic variety $X$ over a field $K$, the set of $K$-\emph{points} of $X$
is the set $X(K)$  consisting of all $K$-scheme morphisms $\Spec(K) \to X$.

\begin{proposition}
\label{p:closed-points-alg-closed-var}
Let $X$ be an algebraic variety over an algebraically closed field $K$.
Then the map from $X(K)$ into $X$, that sends each $f \in X(K)$ to the image by $f$ of the unique point of $\Spec(K)$,
yields a bijection from $X(K)$ onto the set $X_0 \subset X$ consisting of all closed points of $X$.
\end{proposition}

\begin{proof}
See \cite[Corollaire~6.4.2]{grothendieck-ega-1}.
\end{proof}

\begin{remark}
\label{rem:closed-points-rat-points}
Proposition~\ref{p:closed-points-alg-closed-var} allows us,
in the case when  $X$ is an algebraic variety over an algebraically closed field $K$,
 to identify $X(K)$ with   $X_0$.
\end{remark}

\begin{proposition}
\label{p:dim-alg-var}
Let $X$ be an algebraic variety over a field $K$.
Let  $C$ and $D$ be constructible subsets of $X$.
Then the following hold:
\begin{enumerate}[\rm (i)]
\item
the scheme $X$ is Noetherian;
\item
$X$ is Jacobson; 
\item
 $\dim(X_0) = \dim(X) < \infty$;
 \item
$C$ is Jacobson;
\item
$C_0 = C \cap X_0$; 
\item
$\dim(C_0) = \dim(C) = \dim(\overline{C})$;
\item
if $C \subset D$ then $C_0 \subset D_0$.
\end{enumerate}
\end{proposition}

\begin{proof}
Assertions (i) and (ii)  follow for instance from Assertions~(i) and~(iii)
in \cite[Lemma~3.4]{cscp-alg-ca}. 
\par
Since $X$ is Jacobson, we have that $\dim(X_0) = \dim(X)$ by 
Corollary~\ref{c:dim-closed-points-jacob}.
To prove that $\dim(X) < \infty$, as every scheme is locally affine,
we can assume, by virtue of  Proposition~\ref{p:dim-subsets}.(iii), that $X$ is affine.  
Then  $X = \Spec(R)$ for some finitely generated $K$-algebra $R$.
By the Noether normalization lemma, there exist an integer $d \geq 0$ and an injective 
$K$-algebra morphism $K[t_1,\dots,t_d] \to R$ such that $R$ is a finitely generated 
$K[t_1,\dots,t_d]$-module.
This implies $\dim(X) = \dim(R) = d < \infty$ (see \cite[Corollary~5.17]{gorz})
and completes the proof of (iii).
\par
Assertions~(iv) and~(v) follow  from (i) and Proposition~\ref{p:construct-inJacob}.
\par
From (i) and Proposition~\ref{p:construct-inJacob}.(iii), we deduce that
$\dim(C \cap X_0) = \dim(C)$.
Thus, to complete the proof of~(vi), it remains only to show that $\dim(C) = \dim(\overline{C})$.
To see this, we first observe that $C$ contains an open dense subset $U$ of $\overline{C}$ since $C$ is constructible.
Let us equip $\overline{C} \subset X$ with its induced reduced closed subscheme structure.
Then $U$ is an open subscheme of $\overline{C}$ 
and both  $\overline{C}$ and $U$ are $K$-algebraic varieties.
Since $U$ is Noetherian, it admits finitely many irreducible components.
Let $x_1,\dots, x_n$ denote the generic points of the irreducible components of $U$ and
consider their closures $\overline{\{x_1\}}, \dots, \overline{\{x_n\}}$  in $X$,
equipped with their induced reduced closed subscheme structure.
As $U$ is dense in $\overline{C}$, we have that
$\overline{C} = \bigcup_{1 \leq i\leq n} \overline{\{x_i\}}$.
Since each  $\overline{\{x_i\}}$ is a closed irreducible subset of $\overline{C}$, we deduce from
Proposition~\ref{p:dim-subsets}.(v) 
that
\begin{equation}
\label{e:1}
\dim(\overline{C})= \max_{1 \leq i \leq n} \dim(\overline{\{x_i\}}).
\end{equation}
Now observe that, for all $1 \leq i \leq n$,  
the set $U \cap \overline{\{x_i\}}$ is an open subset of  $\overline{\{x_i\}}$
that  is non-empty 
since $x_i \in U$. 
Hence, Theorem~5.22.(3) in~\cite{gorz} applied to the irreducible algebraic varieties $\overline{\{x_i\}}$ implies that 
\begin{equation}
\label{e:2}
\dim(\overline{\{x_i\}}) = \dim(U \cap \overline{\{x_i\}}) \leq \dim(U), 
\end{equation}
where the last inequality follows from Proposition~\ref{p:dim-subsets}.(i).
We deduce from~\eqref{e:1}, \eqref{e:2}, and Proposition~\ref{p:dim-subsets}.(i) that $\dim(\overline{C}) \leq \dim(U) \leq \dim(C)$. 
As $\dim(C) \leq \dim(\overline{C})$ by Proposition~\ref{p:dim-subsets}.(i),
we conclude that
\begin{equation}
\label{e:equal-dimU-C-Cbar}  
\dim(U) = \dim(C) = \dim(\overline{C}).
\end{equation} 
This completes the proof of (vi).
\par
Assertion (vii) is an immediate consequence of (v). 
\end{proof}

\begin{proposition}
\label{p:dim-image}
Let $X$ and $Y$ be  algebraic varieties over a field $K$ and let $f \colon X \to Y$ be a $K$-scheme morphism.
Let $C$ be a constructible subset of $X$.
Then the following hold:
\begin{enumerate}[\rm (i)]
\item 
$f(C)$ is a constructible subset of $Y$;
\item
$f(C_0) = (f(C))_0$;
\item
 $\dim(f(C)) \leq \dim(C)$;
 \item 
$f(X_0) \subset Y_0$;
 \item
 $\dim(f(X)) \leq \dim(X)$;
 \item
 if $E$ is a constructible subset of $X_0$, then $f(E)$ is a constructible subset of $Y_0$ and one has
 $\dim(f(E)) \leq  \dim(E)$.
 \end{enumerate}
\end{proposition}

\begin{proof}
Assertion~(i) is Chevalley's theorem (see e.g.~\cite[Th\'eor\`eme~1.8.4]{grothendieck-20-1964}, 
\cite[p.~93]{hartshorne},  
\cite[Theorem~7.4.2]{vakil}).
\par
For~(ii),
see e.g.~\cite[Lemma~3.6.(v)]{cscp-alg-ca}.
\par
To prove (iii), first observe that $D \coloneqq f(C)$ is a constructible subset of $Y$ by (i).
Therefore $D$ contains a dense open subset $V$ of $\overline{D}$.
Let $y_1,\dots, y_m$ denote the generic points of the irreducible components of $V$
(see the proof of Proposition~\ref{p:dim-alg-var}.(vi)).
As $V \subset D$, 
there exist points $x_1,\dots, x_m \in C$ such that $f(x_i)=y_i$ for
$1 \leq i \leq m$.
 Consider the closure  $\overline{\{x_i\}}$ (resp.~$\overline{\{y_i\}}$)
of the singletons $\{x_i\}$ (resp.~$\{y_i\}$) in $X$ (resp.$Y$),
equipped with their induced reduced closed subscheme structures.
For $1 \leq i \leq m$, 
let $f_i\colon \overline{\{x_i\}} \to \overline{\{y_i\}}$ be the dominant $K$-scheme morphism induced by $f$ (cf.~\cite[Proposition~I.5.2.2]{grothendieck-ega-1}). 
It follows from Theorem~5.22.(3) in~\cite{gorz} that 
\begin{equation}
\label{e:dim-x-i-geq-dim-y-i}
\dim(\overline{\{y_i\}}) \leq \dim(\overline{\{x_i\}}) \text{  for all } 1 \leq i \leq m.
\end{equation}
On the other hand, we have that
$\overline{D} = \bigcup_{1 \leq i\leq m} \overline{\{y_i\}}$
since $V$ is dense in $\overline{D}$,
so that
\begin{equation}
\label{e:dim-clos-image-C}
\dim(\overline{f(C)}) = \max_{1 \leq i \leq n} \dim(\overline{\{y_i\}})
\end{equation}
by applying Proposition~\ref{p:dim-subsets}.(v).
From~\eqref{e:dim-clos-image-C} and~\eqref{e:dim-x-i-geq-dim-y-i}, we get
\begin{equation}
\label{e:ineq-clos-image}
\dim(\overline{D}) \leq \max_{1 \leq i \leq n} \dim(\overline{\{x_i\}}) \leq \dim(\overline{C}), 
\end{equation}
where the last inequality follows from Proposition~\ref{p:dim-subsets}.(i).
Now, as $C \subset X$ and $D \subset Y$ are constructible subsets,
we have  that  $\dim(C) = \dim(\overline{C})$ and 
$\dim(D) = \dim(\overline{D})$ by Proposition~\ref{p:dim-alg-var}.(vi).
Therefore, inequality~\eqref{e:ineq-clos-image} gives us $\dim(D) \leq \dim(C)$.
This completes the proof of~(iii).   
\par
Assertions~(iv) and~(v)
are deduced from~(ii) and ~(iii) 
after taking $C = X$.
\par
Suppose now that  $E$ is a constructible subset of $X_0$.
 Then $E = C \cap X_0$ for some constructible subset $C \subset X$ 
 by Corollary~\ref{c:constr-Jacob-bije}.
We then have  $f(E) = f(C) \cap Y_0$ by virtue of (i), (ii), and 
Proposition~\ref{p:dim-alg-var}.(v),
and hence
\[ 
\dim(f(E))  = \dim(f(C) \cap Y_0) = \dim(f(C)) \leq \dim(C) = \dim(C \cap X_0) = \dim(E)
\]
by using (i), (iii), and Proposition~\ref{p:dim-alg-var}.(vi). This shows (vi).
\end{proof}

\begin{proposition}
\label{p:dim-fiber}
Let $X$ and $Y$ be  algebraic varieties over a field $K$ and let $f \colon X \to Y$ be a $K$-scheme morphism.
For $x \in X$, let $y \coloneqq f(x)$.
Then the following hold:
\begin{enumerate}[\rm (i)]
\item 
there exists a closed point $x \in X$ such that
\begin{equation}
\label{e:dim-fiber}
\dim(f^{-1}(y)) \geq \dim(X)-\dim(Y); 
\end{equation}
\item
if $X$ and $Y$ are both irreductible, then
Inequality~\eqref{e:dim-fiber} is satisfied for every closed point $x \in X$.
\end{enumerate}
\end{proposition}

\begin{proof}
Consider the \emph{geometric fiber} of $f$ at $y$, 
that is, the $Y$-fibered product
$X_y \coloneqq X \times_Y \kappa(y)$, where $\kappa(y)$ is the residue field of $Y$ at $y$,
and recall that the first projection morphism $X_y \to X$ induces a homeomorphism from $X_y$ onto $f^{-1}(y)$ 
(cf. \cite[Proposition~3.1.16]{liu-alg-geom}).
As $X$ and $Y$ are Noetherian schemes,
it follows from Theorem~4.3.12 in~\cite{liu-alg-geom} that
\begin{equation}
\label{e:ineq-geom-fiber-dim}
\dim(\OO_{X_y,x}) \geq \dim(\OO_{X,x})-\dim(\OO_{Y,y})
\end{equation}
for all $x \in X$.
\par
Suppose first that $X$ and $Y$ are irreducible.
If $x$ is a closed point of $X$, then $y=f(x)$ is a closed point of $Y$
(see e.g.~Lemma~3.6 in~\cite{cscp-alg-ca}).  
By applying Corollary~2.5.24 in~\cite{liu-alg-geom}, we then get
\[
\dim(\OO_{X,x})=\dim(X) \quad \text{and} \quad \dim(\OO_{Y,y})=\dim(Y),
\]
so that Assertion~(ii) follows
from~\eqref{e:ineq-geom-fiber-dim} and  the general fact that
$\dim(\OO_{X_y,x})\leq \dim(X_y)=\dim(f^{-1}(y))$.
\par
To prove  Assertion~(i), consider an irreducible component $Z$ of $X$ such that $\dim(Z)=\dim(X)$ and the closure  $V=\overline{f(Z)} \subset Y$ of its image. 
As the closure of every irreducible subset is itself irreducible,  $V$ is also irreducible. 
We equip $Z$ and $V$ with their induced reduced closed subscheme structures and denote by 
$\iota \colon Z \to X$ the  closed immersion associated with $Z$. 
By \cite[Proposition~I.5.2.2]{grothendieck-ega-1}, $f\circ \iota$ induces a $K$-morphism of irreducible algebraic varieties $h \colon Z \to V$. 
Let $x \in Z$ be a closed point and $y=h(x)=f(x)$. 
Then by what we proved above for the irreducible case (Assertion~(ii)), we conclude that  
\begin{equation*}
\dim(f^{-1}(y)) \geq \dim (h^{-1}(y)) 
 \geq \dim(Z)-\dim(V) \geq \dim(X) - \dim(Y),
\end{equation*}
where the first inequality follows from the inclusion $f^{-1}(y) \supset h^{-1}(y)$ and 
the last one  from  the inclusion $V \subset Y$. 
   \end{proof}

\begin{proposition}
\label{p:prod-alg-var}
Let $X$ and $Y$ be algebraic varieties over a field $K$ and let $X \times_K Y$ denote their $K$-fibered product.
Then the following hold:
\begin{enumerate}[\rm (i)]
\item
$X \times_K Y$ is a $K$-algebraic variety;
\item
$(X \times_K Y)(K) = X(K) \times Y(K)$;
\item
$\dim(X \times_K Y)= \dim(X) + \dim(Y)$.
\end{enumerate}
\end{proposition}

\begin{proof}
Assertion~(i) follows from Lemma~4.22 in~\cite{gorz}.
\par
Assertion~(ii) is an immediate consequence of the universal property of $K$-fibered products.
\par
Assertion~(iii) follows from Proposition~5.37  and Proposition 5.50 in~\cite{gorz}.
\end{proof}

\begin{proposition}
\label{p:prod-alg-var-alg-closed}
Let $X$ and $Y$ be algebraic varieties over an algebraically closed  field $K$ and let $X \times_K Y$ denote their $K$-fibered product.
Let $C$ (resp.~$D$) be a constructible subset of $X$ (resp.~$Y$). 
Then the following hold:
\begin{enumerate}[\rm (i)]
\item
$(X \times_K Y)_0 = X_0 \times Y_0$;
\item
$C_0 \times D_0 \subset (X \times_K Y)_0$;
\item 
the set $C_0 \times D_0 \subset (X \times_K Y)_0 \subset X \times_K Y$ being equipped with the Zariski topology, one has that 
$\dim(C_0\times D_0) =\dim(C_0) +\dim(D_0)$;
\item
if $X$ and $Y$ are irreducible, then $X \times_K Y$ is irreducible.
\end{enumerate}
\end{proposition}

\begin{proof}
Assertion~(i) immediately follows from Remark~\ref{rem:closed-points-rat-points} and Assertions~(i) and~(ii) in Proposition~\ref{p:prod-alg-var}.
\par
Since $C_0 \subset X_0$ and $D_0 \subset Y_0$ by Proposition \ref{p:dim-alg-var}.(ii), 
we have that
\begin{equation*}
C_0 \times D_0 \subset X_0 \times Y_0 = (X \times_K Y)_0
\end{equation*}
by using (i). This shows (ii).
\par
Since $C$ (resp.~$D$) is a constructible subset of $X$ (resp.~$Y$),
it contains an open dense subset $U$ (resp.~$V$) of $\overline{C}$ (resp.~$\overline{D}$).
Let us equip $\overline{C}$ and $\overline{D}$ with their induced reduced closed subscheme structure.
Thus $U$, $V$, $\overline{C}$, and $\overline{D}$ are now viewed as $K$-algebraic varieties. 
By properties of base change, $U \times_K V$ is  an open subscheme of 
$\overline{C} \times_K \overline{D}$, 
which is in turn a closed subscheme of $X\times_K Y$.
By applyig Proposition~\ref{p:dim-alg-var}.(vii), we have that
$U_0 \subset C_0 \subset \overline{C}$ and $V_0 \subset D_0 \subset \overline{D}$, so that
\begin{equation}
\label{e:C0-D0-between}
(U \times_K V)_0 = U_0 \times V_0 \subset C_0 \times D_0 \subset (\overline{C})_0 \times (\overline{D})_0 = (\overline{C} \times_K \overline{D})_0 \subset \overline{C} \times_K \overline{D},
\end{equation}
where the two equalities follow from (i).
By using Proposition~\ref{p:dim-subsets}.(i), we deduce from~\eqref{e:C0-D0-between} that
\begin{equation}
\label{e:dim-C0-D0-between}
\dim((U \times_K V)_0) \leq \dim(C_0 \times D_0) \leq \dim(\overline{C} \times_K \overline{D}).
\end{equation}
Now since
\begin{align*}
\dim((U \times_K V)_0) &= \dim(U \times_K V) && \text{(by Proposition~\ref{p:dim-alg-var}.(iii))} \\
&= \dim(U) + \dim(V) && \text{(by Proposition~\ref{p:prod-alg-var}.(iii))} \\
&= \dim(C) + \dim(D) && \text{(by \eqref{e:equal-dimU-C-Cbar})} \\
&= \dim(C_0) + \dim(D_0) && \text{(by Proposition~\ref{p:dim-alg-var}.(vi))}
\end{align*}
and
\begin{align*}
\dim(\overline{C} \times_K \overline{D}) &= \dim(\overline{C}) +  \dim(\overline{D})
&& \text{(by Proposition~\ref{p:prod-alg-var}.(iii))} \\
&= \dim(C_0) + \dim(D_0) && \text{(by Proposition~\ref{p:dim-alg-var}.(vi))}, 
\end{align*}
it follows from~\eqref{e:dim-C0-D0-between} that
$\dim(C_0 \times D_0) = \dim(C_0)+ \dim(D_0)$.
This shows (iii).
\par
Assertion~(iv) follows from Proposition~5.50 in~\cite{gorz}. 
\end{proof}

\begin{remark}
Assertion (iv) in Proposition~\ref{p:prod-alg-var-alg-closed} becomes false if we remove the hypothesis that the field $K$ is algebraically closed.
For example, $X \coloneqq \Spec(\C)$ is an irreducible  algebraic variety over $K \coloneqq \R$ but  
$X \times_K X = \Spec(\C \times \C)$ is not irreducible. 
\end{remark}

\subsection{Projective varieties}

Let $K$ be a field. 
A $K$-algebraic variety $X$ is called a \emph{projective variety} over $K$ if there exists a closed immersion $\iota \colon X \to \Proj^N_K$ for some $N \in \N$ 
(cf.~\cite[Definition~2.3.47]{liu-alg-geom}).  
In that case, we can identify the underlying topological space of $X$ with its image by $\iota$. 

\begin{theorem}[Projective dimension Theorem]
\label{t:proj-dim}
Let $K$ be an algebraically closed field. 
Let $Y,Z$ be closed subschemes of $\Proj^N_K$ of dimension $r,s$ respectively. 
Suppose that $r+s \geq N$. 
Then $Y \cap Z$ is nonempty. 
Equivalently, $Y(K)\cap Z(K) \subset \Proj^N(K)$ is nonempty. 
\end{theorem}

\begin{proof}
Up to replacing $Y$ and $Z$ by irreducible components of maximal dimension equipped  
with their induced reduced closed subscheme structure, 
we can assume that $Y$ and $Z$ are irreducible.  
The theorem is then just a reformulation of Theorem~I.7.1 in~\cite{hartshorne}.  
Indeed, $Y\cap Z$ is closed so it is Jacobson. 
We then equip it with the induced reduced subscheme structure. 
Hence, $Y \cap Z$ is nonempty if and only if it has a closed point, i.e., $(Y\cap Z)(K)=Y(K) \cap Z(K)$ is nonempty. 
See also Proposition~5.40 and its corollaries in~\cite{gorz}. 
\end{proof}

\begin{corollary}
\label{c:proj-dim}
Let $N \in \N $ and let $X \subset \Proj^N_K$ be a projective variety over an algebraically closed field $K$. 
Let $L \subset X$ be a hyperplane section of $X$, i.e., 
$L= H\cap X$, where $H\subset \Proj^N_K$ is a hyperplane not containing $X$. 
Let $C\subset X$ be a closed subscheme such that $\dim(C) \geq 1$. 
Then  $L \cap C$ is nonempty. 
\end{corollary}

\begin{proof}
By hypothesis, $X \subset \Proj^N_K$ is a closed subscheme. 
Hence, $C$ is also a closed subscheme of $\Proj^N_K$ since $C$ is closed in $X$.  
Observe that $\dim(C)+\dim(H)\geq 1+ N-1 = n$. 
Therefore, we deduce from Theorem~\ref{t:proj-dim} that $H\cap C \not= \varnothing$.  
As $C\subset X$, we conclude that  $L\cap C = H\cap X \cap C = H \cap C$ is 
non-empty.  
\end{proof}

\begin{remark}
\label{rem:section-exist}
With the notation as in Corollary~\ref{c:proj-dim}, we claim that hyperplane sections of $X$ always exist. 
Indeed, let $H_0,\dots, H_N$ denote the $N+1$ standard coordinate hyperplanes of $\Proj^N$.  
Since $H_0\cap \dots \cap H_N=\varnothing$ and 
 $X\supset C$ is nonempty, there exists a hyperplane $H_i$ not containing $X$. This proves the claim. 
\par
In fact, let $\iota \colon X\to \Proj^N$ be the closed immersion. 
Let $\OO(1)$ denote  the Serre line bundle of $\Proj^N$. 
Then each global section of the very ample line bundle $\iota^*\OO(1)$ of $X$ is a hyperplane section. 
These global sections, denoted by $H^0(X,\iota^*\OO(1))$, form a strictly positive finite dimensional $K$-vector space. 
See chapters~II.5, III, and Appendix~A in \cite{hartshorne} for more details.   
\end{remark}

Every projective $K$-algebraic variety is $K$-proper
(see for example \cite[Theorem~3.3.30]{liu-alg-geom}). 
The converse is not true. 
However, we have the following consequence of Chow's lemma, 
which we shall use in Section~\ref{sec:weeak-pre-inj}. 

\begin{theorem}
\label{t:chow-lemma}
(Chow's lemma) 
Let $X$ be an irreducible complete algebraic variety over a field $K$. 
Then there exist an irreducible projective $K$-algebraic variety $\tilde{X}$ 
and a surjective $K$-morphism $f \colon \tilde{X} \to X$ with $\dim(\tilde{X})=\dim(X)$. 
\end{theorem}

\begin{proof}
See \cite[Corollaire~II.5.6.2]{grothendieck-ega-2}. 
\end{proof}

\subsection{Amenable groups}
A group $G$ is called \emph{amenable} if there exist a directed set $I$ and a family $(F_i)_{i \in I}$ of non-empty finite subsets of $G$  such that
\begin{equation}
\label{e:folner-s}
\lim_{i} \frac{\vert F_i \setminus  F_i g\vert}{\vert F_i \vert} = 0 \text{  for all } g \in G
\end{equation}
(see \cite[Chapter~4]{ca-and-groups-springer} and the references therein).
Such a family $(F_i)_{i \in I}$ is then called a \emph{(right) F\o lner net} for $G$. 
\par
All finitely generated groups of subexponential growth and all solvable groups are amenable. 
Moreover, the class of amenable groups is closed under the operations of taking subgroups, quotients,  extensions, and direct limits.
On the other hand, every group containing a non-abelian free subgroup is non-amenable.  

\subsection{Tilings}
Let $G$ be a group. 
Let  $E$ and $E'$ be two finite subsets of $G$.
A subset $T \subset G$ is called  an $(E,E')$-\emph{tiling} if
it satisfies the following two conditions:
\begin{enumerate}[\rm (T-1)]
\item 
 the subsets $gE$, $g \in T$, are pairwise disjoint,
 \item
$G = \bigcup_{g \in T} gE'$.
\end{enumerate}
 \par
The following statement is an immediate consequence of Zorn's lemma 
(see \cite[Proposition~5.6.3]{ca-and-groups-springer}).

\begin{proposition} 
\label{p:tilings-exist}
Let $G$ be a group. Let $E$ be a non-empty finite subset of $G$ and let
$E' \coloneqq  EE^{-1} = \{g h^{-1} : g,h \in E\}$. 
Then there exists  an $(E,E')$-tiling $T \subset G$.
\end{proposition}

We shall need  the following estimate on the growth of tilings 
with respect to F\o lner nets. 

\begin{proposition} 
\label{p:folner-tiling}
Let $G$ be an amenable group and let $(F_i)_{i \in I}$ be a  F\o lner net for $G$.
Let $E$ and $E'$ be finite subsets of $G$ and suppose that
$T \subset G$ is an $(E,E')$-tiling. 
For each $i \in I$, define the subset 
$T_i \subset T$ by $T_i \coloneqq  \{g\in T: gE \subset F_i\}$.
Then there exist a real number $\alpha> 0$ and an element $i_0 \in I$ such that 
$| T_i | \geq \alpha | F_i |$
for all $i \geq i_0$. 
\end{proposition}

\begin{proof}
See  \cite[Proposition~5.6.4]{ca-and-groups-springer}. 
\end{proof}
\section{Algebraic cellular automata}
\label{sec:alg-ca}

\subsection{Interiors, neighborhoods, and boundaries}
\label{ss:interiors}
Let $G$ be a group and let $M$ be a finite subset of $G$.
The $M$-\emph{interior} $\Omega^-$ and the $M$-\emph{neighborhood} $\Omega^+$ of a 
subset $\Omega \subset G$ are the subsets of $G$ defined respectively by
\[
\Omega^- \coloneqq  \{g \in G : g M \subset \Omega \},
\]
and
\[
\Omega^+ \coloneqq \Omega M = \{g h  : g \in \Omega \text{ and } h \in M \}. 
\]
Note that $\Omega^- \subset \Omega \subset \Omega^+$ if $1_G \in M$.
\par
We   define the $M$-\emph{boundary} $\partial \Omega$ of $\Omega$ by 
$\partial M = \Omega^+ \setminus \Omega^-$.
\par
If $G$ is an amenable group and $(F_i)_{i \in I}$ is a F\o lner net of $G$, then
one has
\begin{equation}
\label{e:boundary-folner}
\lim_i \frac{|\partial F_i |}{| F_i |} = 0 \quad \text{for every finite subset } M \subset G.
\end{equation}
(see e.g.~\cite[Proposition~5.4.4]{ca-and-groups-springer}).
\par
Let $A$ be a set and let $G$ be a group.
Suppose now that we are given a celular automaton $\tau \colon A^G \to A^G$ with memory set $M$.
Let $\Omega \subset G$ and let $\Omega^-$ and $\Omega^+$ be  defined
as above.\par
The cellular automaton $\tau$ induces maps
$\tau^-_\Omega \colon A^\Omega \to A^{\Omega^{-}}$ and $\tau^+_\Omega \colon A^{\Omega^{+}} \to A^\Omega$ defined respectively by 
\[
\tau^-_\Omega(u) \coloneqq  (\tau(c))\vert_{\Omega^{-}} \quad \text{ for all } u \in A^\Omega,
\]
and 
\[
\tau^+_\Omega(u) \coloneqq  (\tau(c))\vert_\Omega \quad \text{ for all } u \in A^{\Omega^{+}},
\]
where $c \in A^G$ is any configuration extending $u$.
Observe that the maps $\tau^-_\Omega$ and $\tau^+_\Omega$ are well defined.
Indeed, Formula~\eqref{e:local-property} implies that $\tau(c)(g)$ only depends of the restriction of $c$ to $g M$.
\par 

\subsection{Cellular automata over algebraic varieties} (cf.~\cite{cscp-alg-ca})
Let $S$ be a scheme and let $X,Y$ be $S$-schemes.
We  denote by $X(Y)$ the set of $Y$-points of $X$, i.e., the set consisting of all $S$-scheme morphisms $Y \to X$. 
If $E$ is a finite set, $X^E$ will denote the $S$-fibered product of a family of copies of $X$ indexed by $E$. 
Note that $(X^E)(Y) = (X(Y))^E$ by the universal property of $S$-fibered products.
If  $f \colon Z \to X$ is an $S$-scheme morphism,
then $f$ induces a map
$F^{(Y)} \colon Z(Y) \to X(Y)$ given by
$f^{(Y)}(\varphi) = f \circ \varphi$ for all $\varphi \in Z(Y)$.
\par
Let $A \coloneqq X(Y)$ and
let $G$ be a group.
Let $\tau \colon A^G \to A^G$ be a cellular automaton over the alphabet $A$ and the group $G$.
We say that $\tau$ is an \emph{algebraic cellular automaton over the group $G$ and schemes} $S,X,Y$ if 
$\tau$ admits a memory set $M$ such that the associated local defining map $\mu_M \colon A^M \to A$ satisfies  the following condition:

\begin{enumerate}
\item[($*$)]
there exists an $S$-scheme morphism $f \colon X^M \to X$ such that $\mu_M = f^{(Y)}$.
\end{enumerate}
 
 \begin{remark}
 \label{rem:independent-memory}
 If $X(S) \not= \varnothing$, and condition ($*$) is satisfied for some memory set $M$ of $\tau$, then ($*$) is satisfied for any memory set of $\tau$ (see \cite[Proposition~5.1]{cscp-alg-ca}).
This applies in particular when $S = \Spec(K)$ for some algebraically closed field $K$
 since in that case $X(S)$ is (identified with)  the set of closed points of $X$ and $X$ is Jacobson.  
 \end{remark}

\begin{lemma}
\label{l:local-map}
Let $S$ be a scheme and let $X,Y$ be $S$-schemes.
Let $A \coloneqq X(Y)$ and let $G$ be a group.
Suppose that  $\tau \colon A^G \to A^G$ is an algebraic cellular automaton over the schemes $S,X,Y$ and let $M$ be a memory set of $\tau$ satisfying $(*)$.
Let $\Omega$ be a finite subset of $G$.
Then  there exist $S$-scheme morphisms  
$f^-_\Omega \colon X^\Omega \to X^{\Omega^{-}}$ and 
$f^+_\Omega \colon X^{\Omega^{+}} \to X^\Omega$ such that $f_\Omega^{-(Y)}=\tau^-_\Omega$ and $f_\Omega^{+(Y)}=\tau^+_\Omega$. 
\end{lemma}

\begin{proof}
We prove the assertion for $f^-_\Omega$. The construction of $f^+_\Omega$ is similar.
For every $g \in \Omega^{-}$, we consider the $S$-scheme projection morphism 
$p_g \colon X^\Omega \to X^{g M}$  
and the $S$-scheme isomorphism $i_g \colon X^{gM} \to X^M$  induced by the bijective map $g M \to  M$ given by left multiplication by $g^{-1}$. 
Then the family of $S$-scheme morphisms $f \circ i_g \circ p_g \colon X^{\Omega} \to X$ for $g \in \Omega^{-}$, 
 yields, by the universal property of $S$-fibered products,
 a $S$-scheme morphism $f^-_\Omega \colon X^\Omega \to X^{\Omega^{-}}$. 
It is clear from this construction that $f_\Omega^{-(Y)}=\tau^-_\Omega$. 
\end{proof}

Let $G$ be a group and let $K$ be a field. Let $X$ be a $K$-algebraic variety and let 
$A \coloneqq X(K)$.
We say that a cellular automaton $\tau \colon A^G \to A^G$ is an \emph{algebraic cellular automaton over} $(G,X,K)$
if $\tau$ is an algebraic cellular automaton over the group $G$ and the schemes $(K,X,K)$, i.e.,
for some, or equivalently any (by Remark~\ref{rem:independent-memory}), memory set $M$ of $\tau$,
there exists a  $K$-scheme morphism  $f \colon X^M \to X$ such that $f^{(K)} \colon A^M \to A$ is the local defining map of $\tau$ associated with $M$.
\par
Suppose now that the field $K$ is algebraically closed.
Recall from Remark~\ref{rem:closed-points-rat-points}, that $A$ is regarded  as the set of closed points of $X$.
Given a  finite subset $\Omega$ of $G$, we denote by $X^\Omega$ the $K$-fibered product of a family of copies of $X$ indexed by $\Omega$.
It follows from Assertion~(ii) in Proposition~\ref{p:prod-alg-var} that
\[
A^\Omega = (X(K))^\Omega = X^\Omega(K).
\]
Thus, $A^\Omega$ is the set of closed points of the algebraic variety $X^\Omega$.
Note that Proposition~\ref{p:dim-alg-var} and Assertion (iii) in 
Proposition~\ref{p:prod-alg-var} imply that
\begin{equation}
\label{e:kdim-a-to-omeg}
\dim(A^\Omega) = \dim(X^\Omega) = |\Omega| \dim(X) < \infty.
\end{equation}
 In what follows, every subset  of $A^\Omega$, or more generally of $X^\Omega$, 
is  equiped with the topology induced by the Zariski topology on $X^\Omega$. 

\begin{remark}
\label{r:reduce}
Let $\iota \colon X_{red} \to X$ denote the reduced scheme associated to the $K$-algebraic variety $X$. 
Then $X_{red}$ is also a $K$-algebraic variety and the immersion $\iota$ induces 
the  identification $X_{red}(K)=X(K)$. 
Moreover, $(X_{red})^\Omega=(X^\Omega)_{red}$ for every finite subset $\Omega \subset G$, 
so that every algebraic cellular automaton over $(G,X,K)$ can be considered as an algebraic cellular automaton over $(G,X_{red},K)$ (cf. \cite[Remark~9.2]{cscp-alg-ca}). 
Hence, there is no loss of generality to assume that $X$ is reduced. 
\end{remark}

\begin{proposition}
\label{p:set-config-zariski}
Let $G$ be a group   
and let $X$ be an algebraic variety over an algebraically closed field $K$. 
Let $A \coloneqq X(K)$ and 
let  $\tau \colon A^G \to A^G$ be an algebraic  cellular automaton over $(G,X,K)$.
Let   $\Gamma \subset A^G$ and let $\Phi \coloneqq \tau(\Gamma)$ denote the image of $\Gamma$ under $\tau$. 
Then the following hold:
\begin{enumerate}[\rm (i)]
\item
if $\Gamma_\Omega$ is a constructible subset of $A^\Omega$ for every finite subset 
$\Omega \subset G$,
then $\Phi_\Omega$ is a constructible subset of $A^\Omega$ for every finite subset $\Omega \subset G$;
\item
if the variety $X$ is complete 
and   $\Gamma_\Omega$ is closed in  $A^\Omega$ for every finite subset 
$\Omega \subset G$,
then $\Phi_\Omega$ is closed in $A^\Omega$ for every finite subset $\Omega \subset G$.
\end{enumerate}
\end{proposition}

\begin{proof}
Let $M$ be a memory set of $\tau$ such that the associated local defining map
$\mu \colon A^M \to A$ is induced by some $K$-scheme morphism  
$f \colon X^M \to X$.
Let $\Omega$ be a finite subset of $G$ and define
$\Omega^+$, $\tau_\Omega^+$, and $f_\Omega^{+}$ as in Subsection~\ref{ss:interiors} and Lemma~\ref{l:local-map}.
 Note  that $\Phi_\Omega=\tau^+_\Omega(\Gamma_{\Omega^{+}})$.
Thus, if $\Gamma_{\Omega^+}$ is a constructible subset of $A^{\Omega^+}$, then 
$\Phi_\Omega = f_\Omega^{+}(\Gamma_{\Omega^+})$ is a constructible subset of $A^\Omega$ by
Proposition~\ref{p:dim-image}.(vi). This shows (i).
\par
Suppose now that the variety $X$ is complete, i.e., proper over $K$. 
Then,   $X^{\Omega^{+}}$ and $X^\Omega$ are also proper over $K$
since fibered products of proper schemes are proper.
As every $K$-morphism between proper $K$-schemes is closed,
 it follows that  $f^+_\Omega \colon X^{\Omega^{+}} \to X^\Omega$ is closed.
Assume now that   $\Gamma_{\Omega^+}$ is closed in $A^{\Omega^+}$.
This means that there exists a closed subset   $F$ of $X^{\Omega^+}$ such that
$\Gamma_{\Omega^+} = A^{\Omega^+} \cap F$ is the set of closed points of $F$.
We then get, by using Proposition~\ref{p:dim-image}.(ii), 
\[
\Phi_\Omega = f_\Omega^{+}(\Gamma_{\Omega^+}) 
= f_\Omega^{+}(A^{\Omega^+} \cap F) = A^\Omega \cap f^+_\Omega(F),     
\]
which implies that $\Phi_\Omega$ is closed in $A^\Omega$. This shows (ii).
\end{proof}

\section{Algebraic mean dimension}
\label{sec:mean-dim}

The definition of algebraic mean dimension we introduce in this section is analogous to that of topological and measure-theoretic entropy, as well as  the various notions of mean dimension introduced by Gromov in~\cite{gromov-esav}.

\begin{definition}
\label{def:mean-dim}
Let $G$ be an amenable group and let $\FF=(F_i)_{i \in I}$ be a F\o lner net for $G$.  
Let $X$ be an algebraic variety over an algebraically closed field $K$ and 
let $A \coloneqq X(K)$ denote the set of $K$-points of $X$. 
The \emph{algebraic mean dimension}  of a subset
$\Gamma \subset A^G$ with respect to  $\FF$ 
is the quantity $\mdim_\FF(\Gamma)$
defined by
\begin{equation}  
\label{e;mdim}
\mdim_\FF(\Gamma) \coloneqq  \limsup_{i \in I} \frac{\dim(\Gamma_{F_i})}{| F_i |},
\end{equation}
where $\dim(\Gamma_{F_i})$ denotes the Krull dimension of $\Gamma_{F_i} \subset A^{F_i} \subset X^{F_i}$ with respect to the Zariski topology
and $|\cdot|$ denotes cardinality.
\end{definition}

Here are some immediate properties of algebraic mean dimension.

\begin{proposition}
\label{p:prop-faciles-mdim}
Let $G$ be an amenable group and let $\FF=(F_i)_{i \in I}$ be a F\o lner net for $G$.  
Let $X$ be an algebraic variety over an algebraically closed field $K$ and 
let $A \coloneqq X(K)$. 
Then the  following hold:
\begin{enumerate}[{\rm (i)}]
\item 
$\mdim_\FF(A^G) = \dim(X)$;
\item
for all subsets $\Gamma,\Gamma' \subset A^G$ such that  $\Gamma \subset \Gamma'$, one has 
$\mdim_\FF(\Gamma)\leq \mdim_\FF(\Gamma')$;
\item
for every subset $\Gamma \subset A^G$, one has 
$\mdim_\FF(\Gamma) \leq \dim( X)$.
\end{enumerate}
\end{proposition}

\begin{proof}
(i) For every $i \in I$, we have that $(A^G)_{F_i} = A^{F_i}$, 
so that 
\begin{align*}
\frac{\dim((A^G)_{F_i})}{| F_i |} &= \frac{\dim(A^{F_i})}{| F_i |} \\
&= \frac{| F_i| \dim(X)}{| F_i |}
&&\text{(by \eqref{e:kdim-a-to-omeg})} \\ 
&= \dim(X).
\end{align*}
It follows that  $\mdim_\FF(A^G) = \limsup_i \dim(X) = \dim(X)$.
\par
(ii) Suppose that $\Gamma \subset \Gamma' \subset A^G$.
 Then, for all $i \in I$, 
we have that $\Gamma_{F_i} \subset \Gamma'_{F_i}$  
and hence 
$\dim(\Gamma_{F_i}) \leq \dim(\Gamma'_{F_i})$ 
by applying Proposition~\ref{p:dim-subsets}.(i).
We deduce that 
\[
\mdim_\FF(\Gamma) = \limsup_{i \in I} \frac{\dim(\Gamma_{F_i})}{| F_i |} \leq \limsup_{i \in I} \frac{\dim(\Gamma'_{F_i})}{| F_i |} = \mdim_\FF(\Gamma').
\]
\par
Assertion (iii) is an immediate consequence of (i) and (ii).
\end{proof}

\section{Algebraic mean dimension and surjectivity}
\label{sec:surjectivity}

\begin{proposition}
\label{p:aut-decroit-mdim} 
Let $G$ be an amenable group and let $\FF=(F_i)_{i \in I}$ be a F\o lner net for $G$.  
Let $X$ be an algebraic variety over an algebraically closed field $K$ and 
let $A \coloneqq X(K)$. 
Let  $\tau \colon A^G \to A^G$ be an algebraic  cellular automaton over $(G,X,K)$. 
Suppose that a subset  
$\Gamma \subset A^G$ satisfies the following property:
for every finite subset $\Omega \subset G$,
the set  $\Gamma_\Omega$ is a constructible subset of  $A^\Omega$ for the Zariski topology.
Then one has 
\[
\mdim_\FF(\tau(\Gamma)) \leq \mdim_\FF(\Gamma).
\]
\end{proposition}

\begin{proof}
Let $\Gamma'\coloneqq  \tau(\Gamma)$. 
Let $M\subset G$ be a memory set of $\tau$ such that the associated local defining map 
$\mu \colon A^M \to A$ satisfies $\mu = f^{(K)}$ 
for some $K$-scheme morphism $f \colon X^M \to X$.
 By Remark~\ref{rem:independent-memory}, 
 up to replacing $M$ by $M \cup \{1_G\}$ if necessary, we can assume that $1_G \in M$.
 Let $\Omega$ be a finite subset of $G$.
 Observe that,
using the notation introduced in Subsection~\ref{ss:interiors} and Lemma~\ref{l:local-map},
\[
\Gamma'_{\Omega^{-}} = \tau^-_\Omega(\Gamma_\Omega) = f^-_\Omega(\Gamma_\Omega).   
\]
Therefore, it follows from Proposition \ref{p:dim-image}.(vi) that 
\begin{equation}
\label{e:maj-Y-sur-int-ome-lin}
\dim(\Gamma'_{\Omega^{-}}) \leq \dim(\Gamma_\Omega).
\end{equation}
\par 
Since $1_G \in M$, we have that $\Omega^{-} \subset \Omega \subset \Omega^+$ and 
thus 
\[
\Gamma'_\Omega \subset \Gamma'_{\Omega^{-}}\times A^{\Omega \setminus
\Omega^{-}} \subset X^{\Omega^-} \times_K X^{\Omega \setminus \Omega^-}=X^\Omega.
\] 
Therefore, we find that 
\begin{align*}
\dim(\Gamma'_{\Omega})
 &\leq \dim(\Gamma'_{\Omega^{-}} \times A^{\Omega \setminus \Omega^{-}}) \quad (\text{by Proposition } \ref{p:dim-subsets}) \\
 &=  \dim(\Gamma'_{\Omega^{-}}) + \dim(A^{\Omega \setminus \Omega^{-}}) \quad (\text{by Proposition \ref{p:prod-alg-var-alg-closed}.(iii)})\\
 &= \dim(\Gamma'_{\Omega^{-}}) + \vert \Omega \setminus \Omega^{-}\vert \dim(A) \quad (\text{by Proposition \ref{p:prod-alg-var-alg-closed}.(iii)}) \\
 &\leq \dim(\Gamma_\Omega) + \vert \Omega \setminus \Omega^{-}\vert \dim(A)\quad (\text{by } \eqref{e:maj-Y-sur-int-ome-lin})
\end{align*}
Since $\Omega \setminus \Omega^{-}\subset \partial \Omega$, 
we deduce that
\[
\dim(\Gamma'_\Omega) \leq \dim(\Gamma_\Omega)+ \vert \partial \Omega\vert \dim(A).
\]
Taking $\Omega = F_i$, this gives us 
\[
\frac{\dim(\Gamma'_{F_i})}{\vert F_i \vert}
\leq \frac{\dim(\Gamma_{F_i})}{\vert F_i \vert} + \frac{\vert \partial F_i  \vert}{\vert F_i \vert} \dim(A).
\]
Since 
\[
\lim_{i \in I} \vert \partial F_i  \vert / \vert F_i \vert = 0
\] 
by~\eqref{e:boundary-folner}, 
we conclude that 
\begin{align*}
\mdim_\FF(\Gamma') &= \limsup_{i \in I} \frac{\dim(\Gamma'_{F_i})}{\vert F_i \vert}\\
&\leq \limsup_{i \in I} \frac{\dim(\Gamma_{F_i})}{\vert F_i \vert} = \mdim_\FF(\Gamma).
\end{align*}
\end{proof}

\begin{lemma}
\label{l:inv-mdim-non-max}
Let $G$ be an amenable group and let $\FF=(F_i)_{i \in I}$ be a F\o lner net for $G$.  
Let $X$ be an irreducible algebraic variety over an algebraically closed field $K$ and 
let $A\coloneqq X(K)$.
Suppose that  $\Gamma \subset A^G$ satisfies the following condition:
\begin{enumerate}[{\rm (C)}]
\item
there exist finite subsets $E,E' \subset G$ and an $(E,E')$-tiling $T \subset G$ such that for all $g \in T$,  
$\Gamma_{gE} \subsetneqq A^{gE}$ 
is a proper closed subset of $A^{gE}$ for the Zariski topology.  
\end{enumerate}
Then one has $\mdim_\FF(\Gamma) < \dim(X)$.
\end{lemma}

\begin{proof}
For each $i \in I$, define, as in Proposition~\ref{p:folner-tiling},  the subset $T_i \subset T$ by
$ T_i \coloneqq  \{g \in T : gE \subset F_i\}$
and set
\[
F_i^* \coloneqq  F_i \setminus \coprod_{g \in T_i} gE,
\]
where $\coprod$ denotes disjoint union.
For all $g \in T$, the set  $\Gamma_{gE} $ is a proper closed subset
of $A^{gE}$ by our hypothesis (C). 
As $A^{gE}$ is irreducible since $X$ is irreducible and $K$ is algebraically closed 
(cf. Proposition~\ref{p:prod-alg-var}.(iv) and Corollary~\ref{c:jacob-irred}), 
it follows from Proposition~\ref{p:dim-subsets}.(ii)  that   
\begin{equation}
\label{e:maj-pi-X-Nj-lin}
\dim(\Gamma_{gE}) \leq \dim(A^{gE}) - 1 = |gE | \dim(A) - 1 =|g E| \dim(X) - 1
\end{equation}
for all $g \in T$.
Now observe that, for each $i \in I$, 
\[
\Gamma_{F_i} \subset A^{F_i^*} \times \prod_{g \in T_i} \Gamma_{g E} \subset X^{F_i^*} \times_K \prod_{g \in T_i} X^{gE}= X^{F_i},
\]
so that
\begin{align*}
 \dim(\Gamma_{F_i}) &\leq \dim(A^{F_i^*} \times \prod_{g \in T_i} \Gamma_{g E}) 
\quad   \text{(by Proposition~\ref{p:dim-subsets}.(i))} \\
 &= \vert F_i^* \vert \dim(A)+ \sum_{g \in T_i} \dim(\Gamma_{gE}) 
\quad  \text{(by Proposition~\ref{p:prod-alg-var-alg-closed}.(iii))} \\
 &\leq \vert F_i^* \vert \dim(A) + \sum_{g \in T_i} ({\vert gE \vert}\dim(A) - 1)
\quad \text{(by~\eqref{e:maj-pi-X-Nj-lin})} \\
 &= \biggl(\vert F_i^* \vert  + \sum_{g \in T_i } \vert gE \vert\biggr)\dim(A) - \vert T_i \vert\\
&= \vert F_i \vert \dim(A) - \vert T_i \vert
= \vert F_i \vert \dim(X) - \vert T_i \vert.
\end{align*}
Now, by virtue of Proposition~\ref{p:folner-tiling}, 
there exist $\alpha > 0$ and $i_0 \in i$ such that $\vert T_i \vert \geq \alpha
\vert F_i \vert$ for all $i \geq i_0$.  
We deduce that, for all $i \geq i_0$, 
\[
\frac{\dim(\Gamma_{F_i})}{\vert F_i \vert} \leq \dim(X) - \alpha.
\]
This implies  that 
\[
\mdim_\FF(\Gamma) = \limsup_{i \in I} \frac{\dim(\Gamma_{F_i})}{\vert F_i \vert} \leq \dim(X)- \alpha < \dim(X).
\]
\end{proof}

\begin{lemma}
\label{l:Ramma-max-mdim}
Let $G$ be an amenable group and let $\FF=(F_i)_{i \in I}$ be a F\o lner net for $G$.  
Let $X$ be an irreducible algebraic variety  over an algebraically closed field $K$ and 
let $A \coloneqq X(K)$.
Suppose that a $G$-invariant subset $\Gamma \subset A^G$ satisfies the following conditions:
\begin{enumerate}[\rm (D1)]
\item
$\Gamma$ is closed in $A^G$ for the prodiscrete topology;
\item
for every finite subset $\Omega$ of $G$, the set $\Gamma_\Omega$ is  closed in $A^\Omega$ for the Zariski topology;
\item
$\mdim_\FF(\Gamma) = \dim(X)$.
\end{enumerate}
Then one has $\Gamma = A^G$.
\end{lemma}

\begin{proof}
We proceed by contradiction.
Suppose that there is a configuration $c \in A^G$ that does not belong to $\Gamma$.
By (D1), the set $A^G \setminus \Gamma$ is an open subset of $A^G$ for the prodiscrete topology. 
Thus, we can find a finite subset $E \subset G$ such that $c\vert_E \notin \Gamma_E$.
This implies that  $\Gamma_E \subsetneqq A^E$.
As $\Gamma$ is $G$-invariant, we have that
\begin{equation}
\label{e:Gamma-gE-not-all} 
\Gamma_{gE} \subsetneqq A^{gE} \quad \text{for all }g\in G.
\end{equation} 
By Proposition~\ref{p:tilings-exist}, there exist a finite subset $E' \subset G$ and an $(E,E')$-tiling 
$T \subset G$.
Since $\Gamma$ satisfies the hypotheses of Lemma~\ref{l:inv-mdim-non-max} 
by (D2) and~\eqref{e:Gamma-gE-not-all},
we deduce that $\mdim_\FF(\Gamma) < \dim(X)$, which contradicts (D3).  
\end{proof}

We can now prove the main result of this section. 

\begin{theorem}
\label{t:sur-dim}
Let $G$ be an amenable group and let $\FF$ be a F\o lner net for $G$.  
Let $X$ be an irreducible complete algebraic variety  over an algebraically closed field $K$ and 
let $A \coloneqq X(K)$.
Let $\tau \colon A^G \to A^G$ be an algebraic  cellular automaton 
 over $(G,X,K)$.  
Suppose that     $\mdim_\FF(\tau(A^G)) =\dim(X)$.
Then $\tau$ is surjective.
\end{theorem}

\begin{proof}
Let us check that  $\Gamma \coloneqq  \tau(A^G)$ satisfies the hypotheses of 
Lemma~\ref{l:Ramma-max-mdim}.
Condition (D1), that is,
the fact that $\Gamma$ is closed in $A^G$ for the prodiscrete topology,
 follows from  Theorem~6.1 in~\cite{cscp-alg-ca}.
 On the other hand,
Condition (D2), that is, the fact that $\Gamma_\Omega$ is closed in $A^\Omega$ with respect to the Zariski topology for every finite subset $\Omega \subset G$,  is satisfied by 
Proposition~\ref{p:set-config-zariski}.(ii).
By applying Lemma~\ref{l:Ramma-max-mdim},
we conclude that $\Gamma = A^G$.
This shows that  $\tau$ is surjective.
\end{proof}

\section{Algebraic mean dimension and weak pre-injectivity}
\label{sec:weeak-pre-inj} 

In this section, we introduce two notions of weak pre-injectivity for algebraic cellular automata, namely 
$(*)$-pre-injectivity and $(**)$-pre-injectivity. 
We shall see that these two notions are equivalent under general hypotheses and that they are both implied by pre-injectivity. 
\par
We use the following notation. 
Given a set $A$, a group $G$, a finite subset $\Omega \subset G$, a subset $D \subset A^\Omega$, and an element $p \in A^{G \setminus \Omega}$, we write 
\[
D_p\coloneqq D\times \{p\} = \{c \in A^G : c\vert_\Omega\in D \text{ and } c\vert_{G \setminus \Omega}=p \}.
\]
We say that a subset $\Gamma \subset A^G$ has \emph{finite support} if $\Gamma=D_p$ for some $D,p$ as above.  
\par
Let $\tau \colon A^G \to A^G$ be a cellular automaton over the group $G$ and the alphabet $A$ 
with  memory set $M$. 
Observe that if $\Gamma \subset A^G$ has finite support then $\tau(\Gamma)$ also has finite support. 
Indeed, $\tau(D_p)= R_s$ for some subset $R \subset A^{\Omega^+}$ and $s=s(\tau,p) \in A^{G \setminus \Omega^+}$. 
Suppose now that $X$ is an algebraic variety over an algebraically closed field $K$ and $A=X(K)$. 
Then we write $\dim(D_p)\coloneqq \dim(D)$, where $\dim(D)$ is the Krull dimension of 
$D \subset A^\Omega$ with respect to  the Zariski topology. 
Note  that $\dim(D_p)$ is well-defined. 
Indeed, suppose that $C_q=D_p$ for some $C \subset A^\Lambda$ and $q \in A^{G\setminus \Lambda}$, where  $\Lambda$ is a finite subset of $G$. 
Then clearly $C$ and $D$ are homeomorphic so that
$\dim(C_q) = \dim(C) = \dim(D) = \dim(D_p)$.

\begin{definition}
\label{def:*-**}
Let $G$ be a group and let $X$ be an algebraic variety over an algebraically closed  field $K$.
Let $A\coloneqq X(K)$ and let $\tau \colon A^G
\to A^G$ be an algebraic cellular automaton over $(G,X,K)$. 
\par
We say that $\tau$ is \emph{$(*)$-pre-injective} if
there do not exist a finite subset $\Omega \subset G$
 and a subset $H \subsetneq A^\Omega $ that is closed for the Zariski topology such that
 \[
 \tau((A^\Omega)_p)=\tau(H_p) \quad \text{ for all } p \in A^{G\setminus \Omega}. 
 \]
 \par
 We say that $\tau$ is \emph{$(**)$-pre-injective} if
there does not exist a finite subset $\Omega \subset G$ such that
 \[
 \dim(\tau((A^\Omega)_p)) <\dim(A^\Omega) \quad \text{ for all } p \in A^{G\setminus \Omega}. 
 \]
\end{definition}

\begin{remark}
Let us note that $(*)$-pre-injectivity and $(**)$-pre-injectivity as well as pre-injectivity itself, are \emph{local} properties. 
More precisely, using again the notation of  Definition~\ref{def:*-**} and Subsection~\ref{ss:interiors}, 
let $M$ be a  memory set of $\tau$ such that $1_G \in M$ and $M = M^{-1}$. 
Then  $(*)$-pre-injectivity amounts to saying that, 
for every finite subset $\Omega \subset G$, there exist no proper closed subsets $H\subsetneq A^\Omega$ such that 
 \[
\tau^+_\Omega( A^\Omega \times \{q\})=\tau^+_\Omega(H \times \{q\}) \quad \text{ for all } q \in A^{\Omega^+ \setminus \Omega}. 
 \]
 Similarly, $(**)$-pre-injectivity means  (by Proposition~\ref{p:dim-image}.(iii)) that for every finite subset $\Omega\subset G$, we have 
 \[
 \dim(\tau^+_\Omega(A^\Omega  \times \{q\})) = \dim(A^\Omega) \quad \text{ for some } q \in A^{\Omega^+\setminus \Omega}. 
 \]
 Finally, pre-injectivity means that for every finite subset $\Omega\subset G$ 
 and every $q \in A^{\Omega^{++} \setminus \Omega}$ (where $\Omega^{++}=(\Omega^+)^+$), the restriction of  
\[
 \tau^+_{\Omega^+} \colon A^{\Omega^{++}}  \to A^{\Omega^+}
\]
to $A^\Omega \times \{q\} \subset A^{\Omega^{++}}$ is injective.
 \end{remark}

In order to establish, in the next proposition,  some key relations between pre-injectivity, $(*)$-pre-injectivity, and $(**)$-pre-injectivity, 
we shall use the following general auxiliary result. 

\begin{lemma}
\label{l:chow-**}
Let $X$ be an irreducible complete algebraic variety over an algebraically closed field $K$. 
Then there exists a proper closed subset $H \subsetneq X$ satisfying the following property: 
\begin{enumerate}[{\rm (P)}]
\item
if  $Y$ is a $K$-algebraic variety with $\dim(Y) < \dim(X)$ and $h \colon X\twoheadrightarrow Y$ is a surjective $K$-scheme morphism, then one has $h(H)=Y$.  
 \end{enumerate}
\end{lemma}

\begin{proof}
Since $X$ is irreducible and complete over $K$, it follows from Chow's lemma (cf. 
Theorem~\ref{t:chow-lemma}) 
that there exist an irreducible projective $K$-algebraic variety $\tilde{X}$ with $\dim(\tilde{X})=\dim(X)$ 
and a surjective morphism $f \colon \tilde{X} \to X$ of $K$-schemes. 
Let $\tilde{H}$ be a hyperplane section of the projective variety $\tilde{X}$ 
(cf. Corollary \ref{c:proj-dim} and Remark~\ref{rem:section-exist}).  
Let $H\coloneqq f(\tilde{H})\subset X$. 
As $f$ is a morphism between proper schemes, it is proper and hence closed. 
Thus, $H$ is a closed subset of $X$. 
Since $\tilde{H}\subsetneq X$ is a proper closed subset and $\tilde{X}$ is irreducible, we have 
$\dim(\tilde{H}) < \dim(\tilde{X})$. 
We deduce from Proposition \ref{p:dim-image}.(iii) that 
\[
\dim(X)=\dim(\tilde{X}) > \dim(\tilde{H})\geq  \dim(f(\tilde{H}))= \dim(H).
\]
It follows that $H$ is a proper closed subset of $X$. 
\par
Now let $Y$ and $h \colon X\twoheadrightarrow Y$ be as in the statement of the lemma. 
As $X$ is irreducible, $Y=h(X)$ is also irreducible. 
Consider the surjective composite morphism 
\[
g\coloneqq h \circ f \colon \tilde{X} \to Y.
\] 
By Assertion~(ii) of Proposition~\ref{p:dim-fiber} applied to $g \colon \tilde{X} \to Y$, 
for every closed point $y\in Y$, 
the closed subset $g^{-1}(y)\subset \tilde{X}$ satisfies 
\[
\dim(g^{-1}(y)) \geq \dim(\tilde{X})-\dim(Y)= \dim(X)-\dim(Y)\geq 1.
\]
Hence, we deduce from Corollary~\ref{c:proj-dim} that the closed subset 
$\tilde{H} \cap g^{-1}(y)$ is nonempty for every closed point $y \in Y$. 
Therefore, $f(\tilde{H})$ contains the set of closed points $Y_0$ of $Y$. 
As $g=h\circ f$ and $H=f(\tilde{H})$, it follows that 
\begin{equation}
\label{e:chow}
h(H)=h(f(\tilde{H}))=g(\tilde{H}) \supset Y_0.
\end{equation}
Since $h(H)$ is constructible in $Y$ by Chevalley's theorem, $Y\setminus h(H)$ is also constructible and hence Jacobson (cf. Proposition~\ref{p:dim-alg-var}.(iv)). 
On the other hand,   $Y\setminus h(H)$ does not contain any closed point of $Y$
by \eqref{e:chow}. 
We then deduce from  Proposition~\ref{p:dim-alg-var}.(iv)
that the Jacobson space $Y\setminus h(H)$ has no closed points. 
It follows that $Y\setminus h(H)$ is empty. 
This shows that $h(H)=Y$. 
\end{proof}

\begin{proposition}
\label{p:pre-injectivity-*}
Let $G$ be a group and let $X$ be an algebraic variety over an algebraically closed field $K$. 
Let $A\coloneqq X(K)$ and let $\tau \colon A^G \to A^G$ be an algebraic cellular automaton 
over $(G,X,K)$. 
Then the following hold:
\begin{enumerate}[\rm (i)]
\item
if $\tau$ is pre-injective, then $\tau$  is both $(*)$-pre-injective and $(**)$-pre-injective; 
\item
if $X$ is irreducible and $\tau$ is $(**)$-pre-injective, then $\tau$ is $(*)$-pre-injective; 
\item
if $X$ is irreducible and complete over $K$,  then $\tau$ is $(*)$-pre-injective if and only if 
it is $(**)$-pre-injective.  
\end{enumerate}
\end{proposition}

\begin{proof}
Suppose first that $X$ is irreducible and that $\tau$ is not $(*)$-pre-injective, i.e., 
there exists a finite subset $\Omega \subset G$ and a closed  subset $H\subsetneq A^\Omega$ such that
\begin{equation}
\label{e:tau-not-*} 
 \tau((A^\Omega)_p)=\tau(H_p)  \quad \text{ for all } p \in A^{G\setminus \Omega}. 
\end{equation}
Since $K$ is algebraically closed, we deduce from Proposition~\ref{p:prod-alg-var-alg-closed}.(iv) 
that $X^\Omega$ and hence $A^\Omega$ are irreducible.  
Thus, it follows from~\eqref{e:tau-not-*}  that 
\begin{align*}
 \dim(\tau((A^\Omega)_p)) &= \dim(\tau(H_p)) \\ 
 & \leq \dim(H_p)= \dim(H)   && \text{(by Proposition~\ref{p:dim-image}.(iii))}\\
 & <\dim(A^\Omega)  && \text{(by Proposition~\ref{p:dim-subsets}.(ii))}
\end{align*}
for all $p \in A^{G\setminus \Omega}$.
Therefore $\tau$ is not  $(**)$-pre-injective. This proves (ii). 
\par 
Let $M$ be a memory set of $\tau$ and $f\colon X^M \to X$ a $K$-scheme morphism 
 such that  $f^{(K)} \colon A^M \to A$ is the local defining map associated with $M$.
After  enlarging $M$ if necessary, we can assume $1_G \in M$ and $M = M^{-1}$. 
We use again the notation introduced in Subsection~\ref{ss:interiors} and write 
$\Omega^{++} \coloneqq (\Omega^+)^+$.
Note that $\Omega \subset \Omega^+ \subset \Omega^{++}$ since $1_G \in M$. 
\par
For the proof of (i) and (iii), we shall use the following construction. 
We suppose that $\tau$ is not $(**)$-pre-injective, i.e., 
there exists a finite subset $\Omega \subset G$ such that 
\begin{equation}
\label{e:**}
 \dim(\tau((A^\Omega)_p)) <\dim(A^\Omega) \quad \text{ for all } p \in A^{G\setminus \Omega}. 
\end{equation}
Let $p \in A^{G\setminus \Omega}$ with $\Omega$ as above and a configuration $c\in (A^\Omega)_p$ extending $p$. 
Observe that $\tau(c)\vert_{G\setminus \Omega^+}$ only depends on $p$
(here we use the fact that $gM \subset G \setminus \Omega$ for all $g \in G\setminus \Omega^+$ since $M^{-1}  = M$). 
\par
Consider the following closed immersion induced by $p$:
\[
\iota \coloneqq (\Id_{X^\Omega}, p\vert_{\Omega^{++} \setminus \Omega}) \colon X^\Omega 
= X^\Omega\times_K \prod_{\Omega^{++}\setminus \Omega} \Spec(K) \to X^{\Omega^{++}}. 
\]
Let $Z \coloneqq \iota(X^\Omega) \subset X^{\Omega^{++}}$ be the closed image of $\iota$  
equipped with the reduced closed subscheme structure. 
Let $j\colon Z \to X^{\Omega^{++}}$ be the corresponding closed immersion. 
Since we can assume that $X$ is reduced by Remark~\ref{r:reduce}, it follows from \cite[Proposition~I.5.2.2]{grothendieck-ega-1} that $\iota$ factors through a morphism $\gamma \colon X^\Omega \to Z$. 
Note that $Z$ is homeomorphic to $X^\Omega$. 
Note also that for any subset $\Gamma \subset A^\Omega$, we have 
\begin{equation}
\label{e:gamma}
\gamma(\Gamma)=\Gamma  \times \{p\vert_{\Omega^{++}\setminus \Omega} \} = \Gamma_p\vert_{\Omega^{++}\setminus \Omega}.
\end{equation}
We consider now the  $K$-scheme morphism 
\[
h\coloneqq  f^+_{\Omega^+} \circ j  \colon Z \to X^{\Omega^+},
\] 
where $f^+_{\Omega^+} \colon X^{\Omega^{++}} \to X^{\Omega^+}$ is defined as in Lemma~\ref{l:local-map}. 
Clearly 
\begin{equation}
\label{e:sigma-definition}
\sigma \coloneqq h^{(K)} \colon Z(K) \to A^{\Omega^+}
\end{equation}
is the restriction of  
$\tau_{\Omega^+}^+$ to $Z(K)=A^\Omega \times \{p\vert_{\Omega^{++}\setminus \Omega} \}$. 
\par
Let $Y \coloneqq \overline{\im(h)} \subset X^{\Omega^+}$ be the closure of $\im(h)$ in $X^{\Omega^+}$. 
We equip $Y$ with the induced reduced closed subscheme structure over $K$. 
By \cite[Proposition~I.5.2.2]{grothendieck-ega-1}, the morphism 
$h$ factors through a $K$-scheme morphism
\[
k \colon Z \to Y.
\] 
Observe that $\im(h)$ is constructible in $X^{\Omega^+}$ by Chevalley's theorem. 
We then have the identifications $\sigma (Z(K))=h(Z_0)=\im(h)_0$ 
by Proposition \ref{p:dim-image}.(ii) and Remark \ref{rem:closed-points-rat-points}. 
From Proposition \ref{p:dim-alg-var}.(vi), 
\begin{align*}
\dim(\sigma(Z(K))) & =\dim(\im(h)_0)=\dim(\im(h)) \\
&=\dim( \overline{\im(h)})=\dim(Y).
\end{align*}
Thus, we deduce from Inequality~\eqref{e:**} and the above equalities that  
\begin{align}
\label{e:*-and-**}
\dim(Z) & = \dim(X^\Omega)=\dim(A^\Omega) \\
& >  \dim(\tau((A^\Omega)_p)) = \dim(\tau^+_{\Omega^+}(A^\Omega \times \{p\vert_{\Omega^{++}\setminus \Omega} \})) \nonumber\\
& =\dim(\sigma(Z(K)))=\dim(Y). \nonumber 
\end{align}
\par
In order to show (i), suppose that $\tau$ is pre-injective but not $(**)$-pre-injective. 
Let $\Omega \subset G$, $p\in A^{G \setminus \Omega}$ and the maps $h,k$ be constructed as above. 
Since $\tau$ is pre-injective, the map $\sigma=h^{(K)}$ (cf. \eqref{e:sigma-definition}) is injective. 
As $K$ is algebraically closed, we can identify closed points of $Z$ and $Y$ with $Z(K)$ and $Y(K)$ respectively. 
We deduce (from \cite[Lemma~3.6.(iii)]{cscp-alg-ca} for example) that $h$ and thus $k$ are injective. 
Proposition \ref{p:dim-fiber}.(i) applied to $k \colon Z \to Y$ shows that there exists 
a closed point $b\in Y_0 \subset X^{\Omega^+}$ such that 
\begin{equation*}
\dim(h^{-1}(b))=\dim(k^{-1}(b))\geq \dim(Z)-\dim(Y) \geq 1,
\end{equation*}
where the last inequality follows from \eqref{e:*-and-**}. 
This is a contradiction since $h^{-1}(b)=k^{-1}(b)$ is Jacobson and has at most one closed point by the injectivity of $h$ and $k$. 
This proves that when $\tau$ is pre-injective, it must be $(**)$-pre-injective. 
Since pre-injectivity implies trivially $(*)$-pre-injectivity by the definition, the point (i) is proved.  
\par
We proceed now to the proof of (iii). 
Suppose that $X$ is irreducible and complete over $K$ and that $\tau$ is not $(**)$-pre-injective. 
Let $\Omega \subset G$, $p\in A^{G \setminus \Omega}$, $c\in (A^\Omega)_p$ and the maps $h,k, \gamma$ be as above. 
Observe again that $X^\Omega$ is irreducible by Proposition \ref{p:prod-alg-var-alg-closed}.(iv). 
As $X$ is proper, the varieties $X^\Omega$, $X^{\Omega^+}$ and $Z$ are also proper over $K$. 
Hence $h\colon Z \to X^{\Omega^+}$ is a morphism of proper $K$-schemes. 
Consequently, $h$ is closed, $Y=\im(h)$ and thus $k\colon Z \to Y$ is surjective. 
\par
Since $X^\Omega$ is irreducible and complete, Lemma \ref{l:chow-**} shows that 
there exists a proper closed subset $L\subsetneq X^\Omega$ independent of $p$ satisfying:
\begin{enumerate}[{\rm (P)}]
\item
if $V$ is a $K$-algebraic variety with $\dim(V) < \dim(X^\Omega)$ and $\Phi \colon X^\Omega \twoheadrightarrow V$ is a surjective $K$-scheme morphism, then one has $\Phi(L)=V$. 
 \end{enumerate} 
We consider the set of $K$-points of $L$:
\[
H\coloneqq L(K) \subsetneq A^\Omega.
\] 
Applying Property $(P)$ to the surjective morphism 
$k \circ \gamma \colon X^\Omega \twoheadrightarrow Y$, 
we deduce that $k(\gamma(L))=Y$. 
Since a surjective morphism between $K$-algebraic varieties induces a surjective map 
between their sets of closed points (cf.~\cite[Lemma~2.22.(iv)]{cscp-alg-ca}), 
we deduce that 
\begin{equation}
\label{e:h-**}
h(H\times \{p\vert_{\Omega^{++}\setminus \Omega} \})=k((\gamma(H))=Y(K)=h(A^\Omega\times \{p\vert_{\Omega^{++}\setminus \Omega} \}). 
\end{equation}
Now let $d=h(c\vert_{\Omega^{++}})=\tau(c)\vert_{\Omega^+} \in A^{\Omega^+}$ be identified with a closed point in $Y_0=Y(K)$. 
By \eqref{e:h-**}, there exists $z \in H=L(K)$ such that $h(z\times \{p\vert_{\Omega^{++}\setminus \Omega} \})=d$. 
Let $c' \in H_p \subset A^G $ be a configuration such that $c'\vert_\Omega=z$ and $c'\vert_{G \setminus \Omega}=p$. 
Then we see that 
\[
h(c'\vert_{\Omega^{++}})=h(c\vert_{\Omega^{++}})=d.
\]
Hence, $\tau(c')\vert_{\Omega^+}=\tau(c)\vert_{\Omega^+}=d$. 
As $c'\vert_{G \setminus \Omega}=c\vert_{G \setminus \Omega}=p$, we have $\tau(c')\vert_{G\setminus \Omega^+}=\tau(c)\vert_{G\setminus \Omega^+}$. 
Thus, we find that $\tau(c')=\tau(c)$.   
As $p \in A^{G \setminus \Omega}$ and $c\in (A^\Omega)_p$ are arbitrary and $c'\in H_p$, 
we deduce that 
\[
\tau((A^\Omega)_p)=\tau(H_p) \quad \text{for all } p\in A^{G \setminus \Omega}. 
\]
Therefore, $\tau$ is not $(*)$-pre-injective. 
Hence, $(*)$-pre-injectivity implies $(**)$-pre-injectivity when $X$ is irreducible and complete over $K$. 
Together with (ii), this completes the proof of (iii). 
\end{proof}

\begin{proposition}
\label{p:**-implies-mdim}
Let $G$ be an amenable group and let $\FF=(F_i)_{i \in I}$ be a F\o lner net for $G$.  
Let $X$ be an algebraic variety over an algebraically closed field $K$ and 
let $A \coloneqq X(K)$. 
Let  $\tau \colon A^G \to A^G$ be an algebraic  cellular automaton over $(G,X,K)$. 
Suppose that $\tau$ is $(**)$-pre-injective.
Then one has 
\begin{equation}
\label{e:mdim-max}
\mdim_\FF(\tau(A^G)) = \dim(X). 
\end{equation}
\end{proposition}

\begin{proof}
We proceed by contradiction.
Suppose that~\eqref{e:mdim-max} is not satisfied, i.e.,
\begin{equation}
\label{e:mdim-notmax}
\mdim_\FF(\tau(A^G)) < \dim(X). 
\end{equation}
Let $M$ be a memory set for $\tau$ such that $1_G \in M$. 
Let  $\Gamma \coloneqq  \tau(A^G)$. 
As $\Gamma_{F_i^+}$ is a subset of $\Gamma_{F_i}\times
A^{F_i^+ \setminus F_i}$ and
$\Gamma_{F_i}$ is a constructible subset of $A^{F_i}$ 
(by Assertion~(i) of Proposition~\ref{p:set-config-zariski}), we have that
\begin{align*}
\dim(\Gamma_{F_i^+})
&\leq \dim(\Gamma_{F_i}\times
A^{F_i^+ \setminus F_i}) 
&& \text{(by Proposition~\ref{p:dim-subsets}.(i))}\\ 
&= \dim(\Gamma_{F_i}) 
+  \dim(A^{F_i^+ \setminus F_i})
&& \text{(by Proposition~\ref{p:prod-alg-var-alg-closed}.(iii))}\\ 
&= \dim(\Gamma_{F_i}) 
+ | F_i^+ \setminus F_i | \dim(X) 
&&\text{(by \eqref{e:kdim-a-to-omeg})} \\  
&\leq \dim(\Gamma_{F_i}) + | \partial(F_i) | \dim(X)
&& \text{(since $F_i^+ \setminus F_i \subset \partial F_i$)}.
\end{align*}
Hence, we find that 
\[
\frac{\dim(\Gamma_{F_i^+})}{\vert F_i\vert} \leq \frac{\dim(\Gamma_{F_i})}{\vert F_i\vert} + \frac{\vert \partial(F_i) \vert}{\vert F_i\vert} \dim(X).
\]
The above inequality together with~\eqref{e:mdim-notmax} and~\eqref{e:boundary-folner} show that there exists $i_0 \in  I$ such that
\begin{equation}
\label{e:dimnotmax}
\dim(\Gamma_{F_{i_0}^+}) <  \vert F_{i_0} \vert \dim(X).
\end{equation} 
Observe now that for all $p \in A^{G\setminus F_{i_0}}$, we have that
 $\tau((A^{F_{i_0}})_p)\subset \tau(A^G) \cap  (A^{F_{i_0}^+})_q$ where 
 $q=\tau(\tilde{p})\vert_{G\setminus F_{i_0}^+} \in A^{G\setminus F_{i_0}^+}$ and 
 $\tilde{p}\in A^G$ is an arbitrary configuration that extends $p$. 
 Therefore, we have for all $p\in A^{G\setminus F_{i_0}}$ that 
\begin{align*}
\label{a:1}
\dim(\tau((A^{F_{i_0}})_p)) &\leq \dim(\tau(A^G)\vert_{F_{i_0}^+}) = \dim(\Gamma_{F_{i_0}^+})  & \\
& < \vert F_{i_0} \vert \dim(A)= \dim(A^{F_{i_0}}) & (\text{by } \eqref{e:dimnotmax}). 
\end{align*}
We can thus conclude that $\tau$ is not $(**)$-pre-injective.  
\end{proof}

As described by the next proposition, the converse of Proposition \ref{p:**-implies-mdim} also holds if we replace $(**)$-pre-injectivity by $(*)$-pre-injectivity.  

\begin{proposition}
\label{p:mdim-implies-*-pre}
Let $G$ be an amenable group and let $\FF$ be a F\o lner net for $G$.  
Let $X$ be an algebraic variety over an algebraically closed field $K$. 
Let $A \coloneqq X(K)$ and let $\tau \colon A^G \to A^G$ be an algebraic cellular automaton 
over $(G,X,K)$. 
Suppose that $X$ is irreducible and that 
\begin{equation}
\label{e;mdim-leq-dim-V}
\mdim_\FF(\tau(A^G)) = \dim(X).
\end{equation}
Then $\tau$ is $(*)$-pre-injective. 
\end{proposition}

\begin{proof}
We proceed by contradiction. 
Suppose that the cellular automaton $\tau$ is not $(*)$-pre-injective. 
Thus, there exist a finite subset $E\subset G$ and a closed proper subset $H \subsetneq A^E$ such that \begin{equation}
\label{e:dim-hyperplan}
\tau((A^E)_p)=\tau(H_p) \quad \text{ for all } p \in A^{G\setminus E}.
\end{equation}
By Proposition~\ref{p:tilings-exist}, we can find a finite subset $E' \subset G$
such that $G$ contains an
$(E,E')$-tiling $T$. 
For every $t \in T$, we define $H_t \subset A^{tE}$ to be the image of $H$ under the 
canonical bijective map $A^E \to A^{t E}$ that is induced by the left-multiplication by $t^{-1}$. 
Since $\tau$ is $G$-equivariant, we deduce from \eqref{e:dim-hyperplan} that 
for each $t\in T$, we have that 
\begin{equation}
\label{e:dim-hyperplan-g}
\tau((A^{tE})_p)=\tau((H_t)_p) \quad \text{ for all } p \in A^{G\setminus tE}.
\end{equation}
\par
Consider the subset $\Gamma \subset A^G$ defined by
\[
\Gamma \coloneqq A^{G\setminus TE} \times \prod_{t \in T}H_t .
\]
We claim that $\tau(A^G) = \tau(\Gamma)$.  
Indeed, let $c \in A^G$ be any configuration
and let us show that there exists a configuration in $\Gamma$ whose image under $\tau$ is equal to 
$\tau(c)$. 
\par
To see this, consider the set $\Phi \subset A^G$ consisting of all configurations $d \in A^G$ satisfying the following conditions:
\begin{enumerate}[(C1)]
\item
$d\vert_{G \setminus TE} = c\vert_{G \setminus TE}$;
\item
if $t \in T$, then $d\vert_{tE} = c\vert_{tE}$ or $d\vert_{tE} \in H_t$;
\item
$\tau(d) = \tau(c)$.
\end{enumerate}
Given a configuration $d \in \Phi$, we define the subset $T_d \subset T$ by
\[
T_d \coloneqq  \{ t \in T : d\vert_{tE} \in H_t \}.
\]
We partially order $\Phi$ by the relation $\leq$ defined by
\[
d \leq e \iff \left( T_d \subset T_e \text{ and } d\vert_{tE} = e\vert_{tE} \text{  for all } t \in T_d \right).
\]
Let us check that  $\Phi$ satisfies the hypotheses of Zorn's lemma.
The set  $\Phi$ is not empty since $c \in \Phi$.
On the other hand, suppose  that $\Psi$ is a non-empty totally ordered subset of $\Phi$.
Let us show that $\Psi$ admits an upper bound in $\Phi$.
To see this, first observe that, for $t \in T$ fixed, the restrictions $d\vert_{tE}$, $d \in \Psi$, are eventually constant,
i.e., there exists $\lambda_t \in A^{tE}$ such that
$d\vert_{tE} = \lambda_t$ for all $d \in \Psi$ large enough (with respect to $\leq$).
Consider now the configuration $e \in A^G$ defined by
\[
e\vert_{G \setminus TE} = c\vert_{G \setminus TE} \text{   and }
e\vert_{tE} = \lambda_t \text{ for all } t \in T.
\]
It is clear that  $e$ satisfies (C1) and (C2).
If $\Omega$ is a finite subset of $G$, then there are only finitely many $g \in T$ such that $g E \subset \Omega$.
It follows that there exists $d \in \Psi$ such that $e\vert_\Omega = d\vert_\Omega$.
Taking $\Omega = g M$, where $g \in G$ and $M$ is a memory set of $\tau$,
we deduce that $\tau(e)(g) = \tau(d)(g) = \tau(c)(g)$ for all $g \in G$.
This shows that $e$ also satisfies (C3).
Thus, $e \in \Phi$ is an upper bound for $\Psi$.
By Zorn's lemma, $\Phi$ admits a maximal element $m$.
We have that $\tau(m) = \tau(c)$ since $m \in \Phi$ satisfies (C3).
We also have  that $m \in \Gamma$. 
Indeed, otherwise, there would be some $t \in T$ such that $m\vert_{t E} \notin H_t$. But then using~\eqref{e:dim-hyperplan-g}, we could modify $m$ on $t E$
and get an element $m' \geq  m$ in $\Phi$  with $T_{m'} = T_m \cup \{t\}$, contradicting the maximality of $m$.
This completes the proof that $\tau(A^G) = \tau(\Gamma)$.   
\par
We then get
\begin{align*}
\mdim_\FF(\tau(A^G)) &= \mdim_\FF(\tau(\Gamma)) \\
&\leq \mdim_\FF(\Gamma) && \text{(by Proposition~\ref{p:aut-decroit-mdim})} \\ 
&< \dim(X) && \text{(by Lemma~\ref{l:inv-mdim-non-max})} ,
\end{align*}
which contradicts~\eqref{e;mdim-leq-dim-V}.
Observe that the hypotheses of Proposition~\ref{p:aut-decroit-mdim} are satisfied since   
$\Gamma_\Omega$ is a closed and hence constructible subset of $A^\Omega$ for every finite subset $\Omega \subset G$. 
Note also that the hypotheses of Lemma~\ref{l:inv-mdim-non-max} are satisfied 
since $X$ is assumed to be irreducible and  $\Gamma_{tE}=H_t$   is a proper closed subset of $A^{tE}$ for all $t\in T$. 
\end{proof}

\section{Main results}
\label{sec:main-results}
 
 \begin{proof}[Proof of Theorem~\ref{t:main-theorem-3}] 
The result  follows from Proposition~\ref{p:prop-faciles-mdim}.(i) 
and Proposition~\ref{p:mdim-implies-*-pre}. 
\end{proof}

The following statement contains Theorem~\ref{t:main-theorem-1}
as well as Theorem~\ref{t:main-theorem-2}.

\begin{theorem}
\label{t:goe-alg}
Let $G$ be an amenable group. 
Let $X$ be an irreducible complete algebraic variety over an algebraically  closed field $K$ 
and let $A \coloneqq X(K)$. 
Suppose that  $\tau \colon A^G \to A^G$ is  an algebraic cellular automaton over $(G,X,K)$.
Then the following conditions are equivalent: 
\begin{enumerate}[\rm(a)]
\item
$\tau$ is surjective;
\item
$\tau$ is $(*)$-pre-injective;
\item
$\tau$ is $(**)$-pre-injective;
\item
for some (or equivalently any)  F\o lner net $\FF$  of the group $G$, one has $\mdim_\FF(\tau(A^G)) =\dim(X)$.
\end{enumerate}
Moreover, if $\tau$ is pre-injective then it is surjective. 
\end{theorem}

\begin{proof}
The fact that (a) implies (d) follows from Proposition~\ref{p:prop-faciles-mdim}.(i) and the converse implication from Theorem~\ref{t:sur-dim}. 
Thus, (a) and (d) are equivalent.
We have that (d) implies (b) by Proposition~\ref{p:mdim-implies-*-pre},
and (b) implies (c) by Proposition~\ref{p:pre-injectivity-*}.(ii).
On the other hand, we have that (c) implies (d) by Proposition~\ref{p:**-implies-mdim}.
This shows that conditions (b), (c), and (d) are equivalent.
\par
Finally, the last assertion follows from the fact that pre-injectivity implies $(*)$-pre-injectivity by Proposition~\ref{p:pre-injectivity-*}.(i) and 
the implication (b) $\implies$ (a). 
\end{proof}

Let $G$ be a group et $M \subset G$ be finite subset. 
Let $X$ be an algebraic variety over an algebraically closed  field $K$ and 
let $f \colon X^M \to X$ be a $K$-scheme morphism. 
For each field extension $L/K$, let $X_L\coloneqq X \times_K \Spec(L)$ denote the $L$-algebraic variety obtained by the base change $\Spec(L) \to \Spec(K)$. 
Then we have $X_L(L)=X(L)$. 
We denote by $\tau^{(L)}\colon X(L)^G \to X(L)^G$ 
the algebraic cellular automaton over $(G,X_L, L)$ with memory set $M$ and associated local defining map $f^{(L)}$. 

\begin{theorem}
\label{t:base-change-surj}
With the above notation, suppose in addition that $G$ is amenable, $X$ is irreducible and complete,  and  
$L $ is algebraically closed. 
Then 
$\tau^{(K)}$ is surjective if and only if $\tau^{(L)}$ is surjective.  
\end{theorem}

\begin{proof}
This follows from Theorem~\ref{t:sur-dim} and the invariance of dimension of algebraic varieties under 
base change of the ground field. 
Indeed, let $\Gamma^{(K)}\coloneqq \tau^{(K)}(X(K)^G)$ and $\Gamma^{(L)} \coloneqq \tau^{(L)}(X(L)^G)$.   
Let $\Omega \subset G$ be a finite subset. 
Then we have the identifications 
\begin{align*}
&\Gamma^{(K)}_\Omega=f^{+(K)}_\Omega(X(K)^{\Omega^+})= f^+_\Omega(X^{\Omega^+})_0\\
&\Gamma^{(L)}_\Omega=f^{+(L)}_\Omega(X(L)^{\Omega^+})= (f^+_\Omega\times \Id_L)(X_L^{\Omega^+})_0.
\end{align*}
Thus for all finite subset $\Omega \subset G$, we find that  
\[
\dim(\Gamma^{(K)}_\Omega)=\dim(f^+_\Omega(X^{\Omega^+}))=\dim((f^+_\Omega\times \Id_L)(X_L^{\Omega^+}))=\dim(\Gamma^{(L)}_\Omega). 
\]
Let $\FF=(F_i)_{i\in I}$ be a F\o lner net of $G$. 
Then by the definition of mean dimension, we have that
\[  
\mdim_\FF(\Gamma^{(K)}) \coloneqq  \limsup_{i \in I} \frac{\dim(\Gamma^{(K)}_{F_i})}{| F_i |}= \limsup_{i \in I} \frac{\dim(\Gamma^{(L)}_{F_i})}{| F_i |}\eqqcolon \mdim_\FF(\Gamma^{(L)}). 
\]
We can therefore conclude from Theorem~\ref{t:sur-dim} that $\tau^{(K)}$ is surjective if and only if $\tau^{(L)}$ is surjective. 
\end{proof}

\section{Counterexamples}
\label{sec:counterexamples}

In the following example, we shall see  that   
Theorem~\ref{t:main-theorem-2}, Theorem~\ref{t:sur-dim}, and  Proposition~\ref{p:pre-injectivity-*} become false 
if  we remove the hypothesis that $X$ is  irreducible, even if $X$ is assumed to be 
$0$-dimensional.   

\begin{example}
\label{ex:reducible} 
Let $G$ be a group and $K$  an algebraically closed field.
\par
Suppose that  $X$ is a $K$-algebraic variety with $\dim(X) = 0$.
Then $A \coloneqq  X(K)$ is a finite non-empty set.
Moreover, every map $A \to A$ is induced by some $K$-scheme morphism $X \to X$.
 Conversely, given a finite non-empty set $A$,
 there exists a $0$-dimensional $K$-algebraic variety $X$ such that $X(K) = A$.
 We can take for example the reduced $K$-algebraic variety $X$ obtained by taking the discrete union of a family indexed by $A$ of copies of $\Spec(K)$.
\par
Let $A$ be a finite non-empty set and $X$ a $0$-dimensional $K$-algebraic variety such that 
$X(K) = A$.  
Clearly  the cellular automata over the group $G$ and the alphabet $A$ are precisely the    algebraic cellular automata over $(G,X,K)$.
\par
Now let $\tau \colon A^G \to A^G$ be a 
cellular automaton over the group $G$ and the alphabet $A$. 
By the classical Garden of Eden theorem in~\cite{ceccherini}, 
the surjectivity of $\tau$ is equivalent to its pre-injectivity and is also equivalent to the fact that 
$\tau(A^ G)$ has maximal topological entropy.
Note that it immediately follows from the characterization of pre-injectivity by the absence of a  pair of distinct mutually erasable patterns
(see e.g. \cite[Proposition~5.5.2]{ca-and-groups-springer}) 
that $\tau$  is pre-injective if and only if it is $(*)$-pre-injective.
Observe  also that $\tau$ is always $(**)$-pre-injective. 
In the case when $G$ is amenable with a F\o lner net $\FF$ then $\tau$ satisfies $\mdim_\FF(\tau(A^G)) = \dim(X) = 0$
since
\[
\dim(A^\Omega)=0
\]
for every finite subset $\Omega \subset G$. 
The variety  $X$ is irreducible if and only if $A$ is a singleton.  
Otherwise, there exist cellular automata $\tau \colon A^G \to A^G$ that are not surjective (e.g., the map $\tau \colon A^G \to A^G$ defined by $\tau(c) = c_0$ for all $c \in A^G$, where $c_0 \in A^G$ is some constant configuration).
Such a cellular automaton is $(**)$-pre-injective but not $(*)$-pre-injective.  
\end{example}

The next example shows that we cannot  replace the hypothesis that $X$ is irreducible by the weaker hypothesis that it is connected in
 Theorem~\ref{t:main-theorem-2},    Theorem~\ref{t:sur-dim}, and Proposition~\ref{p:pre-injectivity-*}.  

\begin{example}
Let $G$ be an amenable group and let $\FF$ be a   F\o lner net for $G$.
Let $K$ be an algebraically closed field. 
Consider the projective curve $X$ in $\Proj^2_K$ defined by 
\[
X\coloneqq \text{Proj}(K[u,v,w]/(uv)) \subset \Proj^2_K.
\]
Then $X=L_u \cup L_v \subset \Proj^2_K$  
is the union of the two projective coordinate lines 
\[
L_u \coloneqq  \{u=0\}, \quad L_v \coloneqq \{v=0\} \, \subset \Proj^2_K.
\]
Since $X$ has two irreducible components $L_u$ and $L_v$, it is not irreducible.
However, $X$ is clearly connected. 
In the principal affine chart $\A^2_K=D_+(w)=\{w\neq 0\} \subset \Proj^2_K$, we see that $X$ is given by   
\[
Y=\Spec(K[x,y]/(xy))=I_x \cup I_y\subset \A^2_K, 
\]
where $x=u/w$, $y=v/w$ and $I_x=\{x=0\}$, $I_y=\{y=0\}$. 
Let $h \colon Y \to Y$ be the contraction morphism induced by the morphism of $K$-algebras:
\begin{align*}
K[x,y]/(xy) & \to K[x,y]/(xy) \\
(x,y) & \mapsto (x,0).
\end{align*}
It is clear that $h(Y)=h(I_y)=I_y=\Spec(K[x,y]/(y)) \simeq \A^1_K$. 
By, for example, the valuative criteria of properness (cf. \cite[Theorem II.4.7]{hartshorne}), there is a $K$-scheme morphism 
$f \colon X \to X$ extending $h$ and that  
\[
f((0 \colon 1 \colon 0))=f((1 \colon 0 \colon 0))=(1 \colon 0 \colon 0) \in L_v.
\]
Hence, $f(X)=f(L_v)=L_v$ and thus $\dim(f(X))=\dim(L_v)=1$. 
Clearly   $f$ is not surjective and thus $f^{(K)}$ is not surjective either.
\par 
Now let $A \coloneqq X(K)$ and let $\tau \colon A^G \to A^G$ be the   cellular automaton over 
$(G,X,K)$ with memory set $M=\{1_G\}$ and associated local defining map $f^{(K)}\colon A \to A$.
Observe that $\tau$ is not pre-injective since $f$ is not injective and $M = \{1_G\}$.
Also  $\tau$ is not $(*)$-pre-injective since $f(X)=f(L_v)=L_v$.  
On the other hand,    $\mdim_\FF (\tau(A^G))=1=\dim(X)$ and 
$\tau$ is $(**)$-pre-injective but not surjective.  
\end{example} 

The following example shows that 
 Theorem~\ref{t:main-theorem-2} and  Theorem~\ref{t:sur-dim} do not hold in general for irreducible non-complete algebraic varieties.
 
\begin{example}
Let $G$ be an amenable group and let $\FF$ be a   F\o lner net for $G$.
Let $K$ be an algebraically closed field. 
Let $X$ be an irreducible algebraic variety over $K$ and let $A \coloneqq X(K)$. 
Suppose that $f \colon X \to X$ is a  non-surjective dominant $K$-scheme morphism.
Observe that $f$ is not injective by the Ax-Grothendieck theorem. 
Let $\tau \colon A^G \to A^G$ be the algebraic  cellular automaton over $(G,X,K)$ with memory set $M=\{1_G\}$ and associated local defining map  $f^{(K)}\colon A \to A$. 
\par
Since $f$ is dominant, Chevalley's theorem and Proposition~\ref{p:dim-alg-var}.(vi) imply that  
$\dim(f(X))=\dim(X)$. 
As $K$ is algebraically closed, we have that 
\[
\dim(f(A))=\dim(f(X))=\dim(X). 
\] 
We deduce that  $\mdim_\FF (\tau(A^G))=\dim(X)$. 
It is  clear that $\tau$ is both $(*)$- and $(**)$-pre-injective.  
However, since $f$ is not surjective, $f^{(K)}$ is not surjective either 
(see e.g.~\cite[Lemma~2.22.(iv)]{cscp-alg-ca}). 
Hence, $\tau$ is not surjective as well.
Note also that $f^{(K)}$ is not injective since $f$ is not
(see e.g.~\cite[Lemma~2.22.(iii)]{cscp-alg-ca}).
It follows that  $\tau$ is not pre-injective. 
\par
Here is a class of such couples $(X,f)$. 
Let $X=\A^2=\Spec(K[x,y])$ be the affine plane over $K$ and consider the morphism 
$f\colon \A^2 \to \A^2$  given by the morphism of $K$-algebras
\begin{align*}
K[x,y] & \to K[x,y] \\
(x,y) & \mapsto (x^r,x^sP(y)),
\end{align*}
where $r,s \geq 1$ and $P\in K[y]$ is a non-constant polynomial in $y$. 
It is clear that $f(\A^2)=\A^2 \setminus (\{x=0\} \setminus \{(0,0)\})$. 
Hence $f$ is indeed a non-surjective dominant $K$-scheme morphism. 
This construction can be easily generalized to higher dimensional affine spaces $\A^n$ for $n\geq 2$ by using 
for example the morphisms of $K$-algebras given by
\begin{align*}
K[x_1,\dots,x_n] & \to K[x_1,\dots, x_n] \\
(x_1, \dots,x_n) & \mapsto (x_1^{r_1},x_1^{r_2}P_2(x_2), \dots, x_1^{r_n}P_n(x_n)),
\end{align*}
where $r_1,\dots, r_n \in \N^*$ and $P_2,\dots, P_n$ are nonconstant polynomials in $x_2, \dots ,x_n$ respectively. 
\end{example}

We give now an example with non-trivial minimal memory set showing that 
we cannot omit the hypothesis that $X$ is complete in Theorem~\ref{t:sur-dim}.

\begin{example}
In this example, we take  $G \coloneqq \Z$. Thus  $G$ is amenable. 
Let $X \coloneqq \A^1_K=\Spec(K[t])$ be the affine line over an algebraically closed field $K$ and let 
$A \coloneqq X(K)=K$. 
Let $\tau \colon A^G \to A^G$ be the cellular automaton over $(G,X,K)$ with memory set 
$M=\{0,1\}\subset \Z$ and associated local defining map given by $f^{(K)}\colon X(K)^K \to X(K)$,
 where $f \colon X^M \to X$ is the $K$-scheme morphism induced by the  morphism of 
 $K$-algebras
\begin{align*}
K[t] & \to K[x,y] \\
t & \mapsto xy
\end{align*}
Clearly   $\tau \colon K^\Z \to K^\Z$ is given by the formula: 
\[
\tau(c)(n)=c(n)c(n+1)  \quad \text{for all } c\in K^\Z \text{ and } n \in \Z.
\]
Consider the configuration $d \in K^\Z$ such that  $d(-1)=d(1)=1$ and $d(n)=0$ if $n \in \Z \setminus 
\{-1,1 \}$. 
If there were some configuration $c \in \K^\Z$ such that $\tau(c)=d$,
 then we would have $c(0)c(1)=d(0)=0$. 
This would imply $c(0)=0$ or $c(1)=0$ and hence $d(-1)=0$ or $d(1)=0$, which is a contradiction.   
We deduce that $d$ has no pre-image under $\tau$.
Thus $\tau$ is not surjective. 
\par
We claim that $\mdim(\tau(K^\Z))=1=\dim(X)$. 
Indeed, let $\Gamma \coloneqq \tau(K^\Z)$. 
For each $m \in \N$, let $F_m  \coloneqq [-m, m] \cap \Z \subset \Z$. 
Then $\FF \coloneqq (F_m)_m$ is a  F\o lner sequence for $\Z$. 
Note that $F_m^{+}= [-m, m+1] \cap \Z$.
Consider the $K$-scheme morphism (cf. Lemma~\ref{l:local-map})
\[
f^+_{F_m} \colon X^{F_m^{+}} =\A^{2m+2} \to X^{F_m}=\A^{2m+1}.
\] 
Then $\tau^+_{F_m}=f_{F_m}^{+(K)}$. 
It is immediate that 
\[
\tau^+_{F_m}(c_{-m}, \dots, c_{m+1})=(c_{-m}c_{-m+1}, \dots, c_m c_{m+1}).
\]  
We deduce that the image of $f^+_{F_m}$ contains $\A^{2m+1} \setminus L$ where 
\[
L=V((x_{-m}\dots x_m)) \subset \A^{2m+1}=\Spec(K[x_{-m}, \dots, x_m])
\]
 is the union of the $2m+1$ coordinate hyperplanes given by the equation $x_{-m}\dots x_m=0$. 
Therefore, we have $\dim(\im(f^+_{F_m}))=\dim(\A^{2m+1})$ and thus $\dim(\Gamma_{F_m})=\vert F_m \vert =2m+1$ for all $m \in \N$. 
Hence, we conclude that 
\[
\mdim_\FF(\Gamma)=\limsup_m \frac{\dim(\Gamma_{F_m})}{\vert F_m \vert}=1=\dim(A)
\]
 as claimed.  
Since $d$ is almost equal to the configuration $0 \in K^\Z$ and $\tau(d)=\tau(0)=0$, we see that $\tau$ is not pre-injective. 
It follows from Proposition \ref{p:mdim-implies-*-pre} that $\tau$ is $(*)$-pre-injective. 
\end{example}

The following example shows that
the hypothesis that $G$ is amenable cannot be omitted in 
Theorem~\ref{t:main-theorem-1}. 

\begin{example}
\label{ex:free-not-myhill}
Let $G=F_2$ be the free group of rank $2$ based on the generators $a,b$.
We recall that  $G$ is residually finite but not amenable. 
Let $M \coloneqq \{a,b,a^{-1}, b^{-1}\} \subset F_2$.  
Consider an abelian variety $Y=(Y,+)$ over an algebraically closed field $K$ 
with indentity element $e\in Y(K)$.  
We suppose that $Y$ is non-trivial, so that  $\dim(Y)\geq 1$. 
The $K$-fibered product $X \coloneqq  Y \times_K Y$ is also  a non-trivial abelian variety over $K$.
The set  $A \coloneqq X(K)=Y(K)\times Y(K)$ of $K$-points of $X$ 
is a non-trivial abelian group.  
For $i=1,2$ let $q^i \colon Y^M \times_K Y^M \to Y^M$ be the $i$-th projection 
and for $g \in M$, let $q_g\colon Y^M \to Y$ be the projection on the $g$-factor. 
For $i=1,2$ and $g\in M$, let $p^i_g \colon X^M=Y^M \times_K Y^M \to Y$ 
be the projection defined by $p^i_g \coloneqq q_g \circ q^i$. 
Let $h \colon X^M \to Y$ be the morphism defined by 
\[
h\coloneqq p^1_a +p^1_{a^{-1}}+p^2_b+p^2_{b^{-1}}.
\]
Let $\iota\coloneqq (\Id_Y, e) \colon Y= Y \times_K \Spec(K) \to X=Y \times_K Y$. 
Finally, we define $f \coloneqq \iota \circ h \colon X^M \to X$. 
\par
Let $\tau \colon A^G \to A^G$ be the algebraic cellular automaton over $(G,X,K)$ with memory set $M$ and associated local defining map $\mu= f^{(K)} \colon A^M \to A$. 
Observe that for all $c\in A^G$:
\[
\tau(c)(g)=\left( \pi_1(c(ga))+ \pi_1(c(ga^{-1}))+ \pi_2(c(gb)) + \pi_2(c(gb^{-1})), e\right), 
\]
where $\pi_i \colon A \to A$ 
is given by $\pi_i(u_1,u_2)=(u_i,e)$ for $i=1,2$ and $(u_1, u_2) \in Y(K)\times Y(K)=A$. 
By \cite[Proposition~5.11]{ca-and-groups-springer}, we see that $\tau$ is pre-injective but not surjective.  
\end{example} 

The following example shows in particular that Theorem \ref{t:main-theorem-3} yields another characterization of group amenability. 

\begin{example}
Let $G$ be a group and let $K$ be an algebraically closed field.    
Let $X \coloneqq \A^n_K$ and $A \coloneqq X(K)=K^n$.   
Suppose that $\tau \colon A^G \to A^G$ is a $K$-linear cellular automaton (cf.\ \cite[Section~8.1]{ca-and-groups-springer}).  
Clearly $\tau$ is an algebraic cellular automaton over $(G,X,K)$. 
\par
We claim that $\tau$ is pre-injective if and only if
it is $(*)$-pre-injective.
Suppose that $\tau$ is not pre-injective.  
Hence there exists a configuration $c\in A^G$ with a nonempty finite support $\Omega \subset G$ such that $\tau(c)=0$
(cf.\ \cite[Proposition 8.2.5]{ca-and-groups-springer}).  
Let $M \subset G$ be a memory set for $\tau$ such that $1_G \in M$ and $M = M^{-1}$. 
Set $\Omega^{++} \coloneqq (\Omega^+)^+ = \Omega M^2$ (cf.\ Subsection \ref{ss:interiors}).
Recall that $\Omega \subset \Omega^{+} \subset \Omega^{++}$ since $1_G \in M$.
Let $H$ be a linear hyperplane in $A^{\Omega^{++}}$ not containing $c\vert_{\Omega^{++}}$, so that $A^{\Omega^{++}} = 
H \oplus Kc\vert_{\Omega^{++}}$. 
Let us show that 
\begin{equation}
\label{e:tau-omega-p-egale-tau-H-p}
\tau((A^{\Omega^{++}})_p)=\tau(H_p) \ \ \ \mbox{ for all } p \in A^{G\setminus \Omega^{++}}.
\end{equation}
Let $p\in A^{G\setminus \Omega^{++}}$. 
Since $H_p \subset (A^{\Omega^{++}})_p$, we only need to show the inclusion $\tau((A^{\Omega^{++}})_p) \subset \tau(H_p)$. 
Let $d \in (A^{\Omega^{++}})_p$. Then we can find $h \in H$ and $k \in K$ such that $d\vert_{A^{\Omega^{++}}} = h + kc\vert_{\Omega^{++}}$. 
Let $h_p \in H_p$ be the unique configuration such that $(h_p)\vert_{A^{\Omega^{++}}} = h$ and $(h_p)\vert_{G \setminus A^{\Omega^{++}}} = p$,
and let us show that $\tau(d) = \tau(h_p)$.
Let $g \in G$ and suppose first that $g \in \Omega^+$ so that $gM \subset \Omega^{++}$. 
Then $d \vert_{gM} = (h_p + kc)\vert_{gM}$ and therefore
\[
\tau(d)(g) = \tau(h_p + kc)(g) = \tau(h_p)(g) + k\tau(c)(g) = \tau(h_p)(g),
\]
where the last equality follows from the fact that $\tau(c) = 0$.
Suppose now that $g \in G \setminus \Omega^+$. Then $gM \cap \Omega = \varnothing$ (here we use $M = M^{-1}$) so that 
$d\vert_{gM} = (h_p)\vert_{gM}$ and therefore
\[
\tau(d)(g) =  \tau(h_p)(g).
\]
Thus $\tau(d) = \tau(h_p)$, and the inclusion $\tau((A^{\Omega^{++}})_p) \subset \tau(H_p)$ follows. 
This shows \eqref{e:tau-omega-p-egale-tau-H-p}.
Now, since $H$ is a proper closed subset of $A^{\Omega^{++}}$, we deduce that $\tau$ is not $(*)$-pre-injective.
Since, conversely, every pre-injective algebraic cellular automaton is $(*)$-pre-injective by Proposition \ref{p:pre-injectivity-*}, 
this proves our claim.
\par 
Suppose now that $G$ is non-amenable. By a result of Bartholdi \cite[Corollary 1.5]{bartholdi} there exists an integer $n \geq 1$ and
a surjective $K$-linear cellular automaton $\tau \colon A^G \to A^G$ that is not pre-injective (and hence not $(*)$-pre-injective).
As $X = \A^n_K$ is irreducible, this shows that Theorem \ref{t:main-theorem-3} becomes false if the group $G$ is non-amenable.
\end{example}

\section{Questions} 

\begin{question}
Can we remove the hypothesis that $X$ is complete (respectively, irreducible) in Theorem~\ref{t:main-theorem-1}?
\end{question}

\begin{question}
Does Theorem~\ref{t:main-theorem-3} 
still hold without the assumption that $X$ is irreducible?
\end{question}

\begin{question}
Does Theorem \ref{t:base-change-surj} remain valid without the amenability hypothesis on the group $G$?
\end{question}

\begin{question}
Does Theorem~\ref{t:main-theorem-1} characterize amenable groups?
\end{question}

\bibliographystyle{siam}

\end{document}